\documentclass[twoside,12pt]{article}
\usepackage{amsmath,amssymb,amsfonts}
\newenvironment{proof}{\noindent {\em {Proof}}.}{$\square$
\medskip}
\newtheorem{definition}{Definition}[section]
\newtheorem{corollary}{Corollary}[section]

\newtheorem{proposition}{Proposition}[section]

\newtheorem{lm}{Lemma}[section]
\newtheorem{ca}{Case}[section]

\newtheorem{coj}{Conjecture}[section]
\begin{document}
\setlength{\unitlength}{10mm}
\newcommand{\f}{\frac}
\newtheorem{theorem}{Theorem}[section]
\newcommand{\sta}{\stackrel}
\title{\bf OD-Characterization of Some Linear
Groups Over Binary Field and Their Automorphism
Groups\thanks{This work has been supported by {\bf RIFS}.}}
\author{{\sc A. R. Moghaddamfar} \ and  \ {\sc S. Rahbariyan}
\\[0.3cm]
{\em Department of Mathematics,}\\[0.1cm] {\em K. N. Toosi
University of Technology,}\\[0.1cm]
 {\em P. O. Box $16315$-$1618$, Tehran, Iran}\\[0.3cm]
 {\em and} \\[0.2cm]
{\em Research Institute for Fundamental Sciences (RIFS), Tabriz, Iran}\\[0.3cm]
{\em E-mails}: {\tt moghadam@kntu.ac.ir} and {\tt
moghadam@ipm.ir}\\[0.2cm]}
\maketitle
\begin{abstract} \noindent The Gruenberg-Kegel
graph ${\rm GK}(G)=(V_G, E_G)$ of a finite group $G$ is a simple
graph with vertex set  $V_G=\pi(G)$, the set of all primes
dividing the order of $G$, and such that two distinct vertices $p$
and $q$ are joined by an edge, $\{p, q\}\in E_G$, if $G$ contains
an element of order $pq$. The degree ${\rm deg}_G(p)$ of a vertex
$p\in V_G$ is the number of edges incident on $p$. In the case
when $\pi(G)=\{p_1, p_2, \ldots, p_h\}$ with $p_1< p_2< \cdots <
p_h$, we consider the $h$-tuple $D(G)=({\rm deg}_G(p_1), {\rm
deg}_G(p_2),\ldots, {\rm deg}_G(p_h))$, which is called the degree
pattern of $G$. The group $G$ is called $k$-fold
OD-characterizable if there exist exactly $k$ non-isomorphic
groups $H$ satisfying condition $(|H|, D(H))=(|G|, D(G))$.
Especially, a 1-fold OD-characterizable group is simply called
OD-characterizable. In this paper, we first find the degree
pattern of the projevtive special linear groups over binary field
$L_n(2)$ and among other results we prove that the simple groups
$L_{10}(2)$ and $L_{11}(2)$ are OD-characterizable (Theorem
\ref{10-11}). It is also shown that automorphism groups ${\rm
Aut}(L_p(2))$ and ${\rm Aut}(L_{p+1}(2))$, where $2^p-1$ is a
Mersenne prime, are OD-characterizable (Theorem \ref{auto}).
\end{abstract}
\def\thefootnote{ \ }
\section{Introduction}
Throughout this paper, all groups considered are {\em finite} and
simple groups are {\em non-abelian}. Given a group $G$, denote by
$\pi_e(G)$ the set of order of all elements in $G$. It is clear
that the set $\pi_e(G)$ is {\em closed} and {\em partially
ordered} by divisibility, hence, it is uniquely determined by
$\mu(G)$, the subset of its  maximal elements. We also denote by
$\pi(n)$ the set of all prime divisors of a positive integer $n$.
For a finite group $G$, we shall write $\pi(G)$ instead of
$\pi(|G|)$.

To every finite group $G$ we associate a graph known as {\it
Gruenberg-Kegel graph} (or {\em prime graph}) denoted by ${\rm
GK}(G)=(V_G, E_G)$. For this graph we have $V_G=\pi(G)$, and for
two distinct vertices $p, q\in V_G$ we have $\{p, q\}\in E_G$ if
and only if $pq\in \pi_e(G)$. When $p$ and $q$ are adjacent
vertices in ${\rm GK}(G)$ we will write $p\sim q$. Denote the
connected components of ${\rm GK}(G)$ by ${\rm
GK}_i(G)=(\pi_i(G), E_i(G))$, $i=1, 2, \ldots, s(G)$, where
$s(G)$ is the number of connected components of ${\rm GK}(G)$. If
$2\in \pi(G)$, then we set $2\in \pi_1(G)$. In the papers
\cite{k} and \cite{w} the connected components of the
Gruenberg-Kegel graph of all non-abelian finite simple groups are
determined. An corrected list of these groups can be found in
\cite{km}.

Recall that a complete graph is a graph in which every pair of
vertices is adjacent. It is worth noting that if $S$ is a simple
group with disconnected prime graph, then all connected
components ${\rm GK}_i(S)$ for $2\leqslant i\leqslant s(S)$ are
complete graphs, for instance, see \cite{suz}.

When the group $G$ has connected components ${\rm GK}_1(G), {\rm
GK}_2(G), \ldots, {\rm GK}_{s(G)}(G)$, $|G|$ can be expressed as
the product of $m_1, m_2, \ldots, m_{s(G)}$, where $m_i$'s are
positive integers with $\pi(m_i)=\pi_i(G)$. We call $m_1,
m_2,\dots, m_{s(G)}$ the {\em order components} of $G$ and we
write $${\rm OC}(G):=\{m_1, m_2, \dots, m_{s(G)}\},$$ the set of
all order components of $G$.

The {\em degree $\deg_G(p)$ of a vertex} $p\in \pi(G)$ is the
number of edges incident on $p$. When there is no ambiguity on
the group $G$, we denote $\deg_G(p)$ simply by $\deg(p)$.
 If $\pi(G)$ consists of the primes $p_1, p_2, \ldots,
p_h$ with $p_1<p_2<\cdots<p_h$, then we define
\[ {\rm D}(G):=\big(\deg_G(p_1), \deg_G(p_2), \ldots, \deg_G(p_h)\big), \]
which is called the {\em degree pattern of $G$}. Moreover, we set
$$ \Omega_n(G):=\{p\in
\pi(G)| \ {\rm deg}_G(p)=n\},$$ for $n=0, 1, 2, \ldots, h-1$.
Clearly,
$$\pi(G)= \bigcup_{n=0}^{h-1}\Omega_n(G).$$ Moreover, since
${\rm deg}_G(p)=0$ if and only if $(\{p\}, \emptyset)$ is a
connected component of ${\rm GK}(G)$, we have $
|\Omega_0(G)|\leqslant s(G)\leqslant 6$ (see \cite{w}). A group
$G$ is called a $C_{p,p}$-group if $p\in \Omega_0(G)$.

Given a finite group $M$, denote by $h_{\rm OD}(M)$ the number of
isomorphism classes of finite groups $G$ such that $|G|=|M|$ and
${\rm D}(G)={\rm D}(M)$. In terms of the function $h_{\rm OD}$,
we have the following definition.
\begin{definition} A finite group $M$ is called {\em $k$-fold {\rm
OD}-characterizabale} if $h_{\rm OD}(M)=k$. Usually, a $1$-fold
OD-characterizabale group is simply called {\em {\rm
OD}-characterizabale}.
\end{definition}

The notion of OD-characterizability of a finite group was first
introduced by the first author and his colleagues in \cite{mzd}.
It is well-known that, according to Cayley's theorem, for each
positive integer $n$ there are only {\em finitely} many
non-isomorphic groups of order $n$ normally denoted by $\nu(n)$.
Hence
$$1\leqslant h_{\rm OD}(G)\leqslant \nu(|G|),$$
for every finite group $G$, and the following result follows
immediately.
\begin{theorem}
Every finite group is $k$-fold {\rm OD}-characterizable for some
natural number $k$.
\end{theorem}

For recent results concerning the simple groups which are
$k$-fold {\rm OD}-characterizable, for $k\geqslant 2$, it was
shown in \cite{akbari-moghadam}, \cite{mz2} and \cite{mzd} that
each of the following pairs $\{K_1, K_2\}$ of groups:

\begin{center}
$\begin{array}{ll} \{{A}_{10}, {\Bbb Z}_3\times J_2\}, &  \\[0.2cm] \{B_3(5), \ C_3(5)\}, & \\[0.2cm]
\{B_m(q), \ C_m(q)\}, &  m=2^f\geqslant 2,  \
|\pi\big((q^m+1)/2\big)|=1,
\  q \ {\rm is \ an \ odd \ prime \ power},\\[0.2cm]
\{B_p(3), \ C_p(3)\}, & |\pi\big((3^p-1)/2\big)|=1, \  p \ {\rm is
\ an \ odd \
prime,}  \\
\end{array}$
\end{center}
satisfy $|K_1|=|K_2|$ and ${\rm D}(K_1)={\rm D}(K_2)$, and $h_{\rm
OD}(K_i)=2$. In general, for simple groups $B_m(q)$ and $C_m(q)$
we have $$\big(|B_m(q)|, \ {\rm D}(B_m(q))\big)=\big(|C_m(q)|, \
{\rm D}(C_m(q))\big),$$ (see {\rm \cite[Proposition 7.5]{vv}}).
Notice that the orthogonal group $B_n(q)$ is isomorphic to the
symplectic group $C_n(q)$ when $q$ is even, and also $B_2(q)\cong
C_2(q)$ for each $q$. Hence, if $B_m(q)$ and $C_m(q)$ are
non-isomorphic groups, then it follows that
$$h_{\rm OD}(B_m(q))=h_{\rm OD}(C_m(q))\geqslant 2.$$

Until recently, we do not know if there exists a non-abelian
finite {\em simple} group which is $k$-fold OD-characterizable for
$k\geqslant 3$. Therefore, the following problem may be of interest.\\[0.2cm]
{\bf Problem 1.} \ {\it Is there a non-abelian finite simple group
$S$ for which $h_{\rm OD}(S)\geqslant 3$ $?$}\vspace{0.2cm}

In this paper, we focus our attention on the OD-characterizability
of projevtive special linear groups over binary field, that is
${\rm PSL}(n, 2)$, and their automorphism groups. We shortly
denote ${\rm PSL}(n, q)$ by $L_n(q)$. Recall that $L_2(2)\cong
{\Bbb S}_3$, $L_3(2)\cong L_2(7)$ and $L_4(2)\cong {\Bbb A}_8$.
Clearly $s(L_2(2))=2$. By \cite{k}, we have $s(L_3(2))=3$,
$s(L_4(2))=2$, and
\[ s(L_n(2))=\left \{ \begin{array}{ll}
1  &
\mbox{ if $n\neq p$, $p+1$;}\\[0.2cm]
  2 & \mbox{ if $n=p$ or $p+1$,}
\end{array}
\right.
\]
where $p\geqslant 5$ is a prime number. More precisely, in the
latter case, when $n=p$ or $p+1$, $L_n(2)$ has two connected
components, one of them is ${\rm GK}_1(L_n(2))$ with
$$ \pi_1(L_p(2))=\pi\Big(2\prod_{i=1}^{p-1}(2^i-1)\Big), \ \ (\mbox{resp.} \
\pi_1(L_{p+1}(2))=\pi\Big(2(2^{p+1}-1)\prod_{i=1}^{p-1}(2^i-1)\Big)),$$
and the other in both cases is ${\rm GK}_2(L_n(2))$ with
$\pi_2=\pi (2^p-1)$, while if $n\neq p, p+1$, then
$\pi_1(L_n(2))=\pi(L_n(2))$. The orders of finite simple groups
under discussion here are:
$$|L_n(2)|=2^{n\choose 2}\prod_{i=2}^{n}(2^i-1).$$
Previously, it was proved that many of projevtive special linear
groups over binary field are OD-characterizable.
\begin{itemize}
\item  It was proved in \cite{amr} that the linear groups $L_p(2)$ and $L_{p+1}(2)$,
for which $|\pi(2^p-1)|=1$,  are OD-characterizable. Note that if
$|\pi(2^p-1)|=1$, then $2^p-1$ is a prime (see \cite[Ch. IX, Lemma
2.7]{hupert}). A list of all known primes $p$ such that $2^p-1$
is also prime (which is called a Mersenne prime) is as follows:
2, 3, 5, 7, 13, 17, 19, 31, 61, 89, 107, 127, 521, 607, 1279,
2203, 2281, 3217, 4253, 4423, 9689, 9941, 11213, 19937, 21701,
23209, 44497, 86243, 110503, 132049, 216091, 756839, 859433,
1257787, 1398269, 2976221, 3021377, 6972593, 13466917, 20996011,
24036583, 25964951, 30402457, 32582657, 37156667, 43112609 (see
\cite{mersenne}). Therefore, the linear groups $L_p(2)$ and
$L_{p+1}(2)$ for these primes $p$ are OD-characterizable.
\item  The OD-characterizability of $L_9(2)$ was established in \cite{khoshravi}.
\end{itemize}

For the values of $|G|$, $s(G)$ and $h_{\rm OD}(G)$ for certain
projective special linear groups over binary field, see Table 1.
\begin{center}
{\bf Table 1.} {\em The value of $h_{\rm OD}(\cdot)$ for some
projective special linear groups over binary field.} \\[0.2cm]
\begin{tabular}{lllcc}
\hline $G$ & $|G|$  & $s(G)$ & $h_{\rm OD}(G)$ & Refs. \\[0.2cm]
\hline \\[-0.2cm]
$L_2(2)\cong {\Bbb S}_3$ & $2\cdot 3$ & 2 & 1 &  \cite{mz}\\[0.2cm]
$L_3(2)\cong L_2(7)$ & $2^3\cdot 3 \cdot 7$ & 3 & 1 &  \cite{amr, zshi}\\[0.2cm]
$L_4(2)\cong A_8$ & $2^6\cdot 3^2 \cdot 5\cdot 7$ & 2 & 1 &  \cite{mz}\\[0.2cm]
$L_5(2)$ & $2^{10}\cdot 3^2 \cdot 5\cdot 7\cdot 31$ & 2 & 1 &  \cite{amr}\\[0.2cm]
$L_6(2)$ & $2^{15}\cdot 3^4 \cdot 5\cdot 7^2\cdot 31$ & 2 & 1 &  \cite{amr}\\[0.2cm]
$L_7(2)$ & $2^{21}\cdot 3^4 \cdot 5\cdot 7^2\cdot 31\cdot 127$ & 2 & 1 &  \cite{amr}\\[0.2cm]
$L_8(2)$ & $2^{28}\cdot 3^5 \cdot 5^2\cdot 7^2\cdot 17\cdot 31\cdot 127$ & 2 & 1 &  \cite{amr}\\[0.2cm]
$L_9(2)$ & $2^{36}\cdot 3^5\cdot 5^2 \cdot 7^3\cdot 17 \cdot
31 \cdot 73 \cdot 127$ & 1 & 1 &  \cite{khoshravi} \\[0.2cm]
$L_{10}(2)$ & $2^{45}\cdot 3^6\cdot 5^2 \cdot 7^3\cdot 11\cdot 17
\cdot
31^2 \cdot 73 \cdot 127$ & 1 & Unknown & -\\[0.2cm]
$L_{11}(2)$ & $2^{55}\cdot 3^6\cdot 5^2 \cdot 7^3\cdot 11\cdot 17
\cdot 23\cdot
31^2 \cdot 73\cdot 89 \cdot 127$ & 2 & Unknown & -\\[0.2cm]
\hline
\end{tabular}
\end{center}

So far, we have not found any natural number $n\geqslant 2$ for
which $h_{\rm OD}(L_n(2))>1$. On this basis, we put forward the
following conjecture.
\begin{coj}
The projective special linear groups $L_n(2)$ for all integers
$n\geqslant 2$ are OD-characterizable.
\end{coj}

In this paper, we will continue to review research on this subject
and we show the following result which confirms the above
conjecture.
\begin{theorem}\label{10-11}
The projective special linear groups $L_{10}(2)$ and $L_{11}(2)$
are OD-characterizable.
\end{theorem}

It should be mentioned that, in fact, among the finite simple
groups with disconnected Gruenberg-Kegel graph, $L_{11}(2)$ is a
first example of the simple OD-characterizable group $S$ with
$\Omega_0(S)=\emptyset$, whereas for all the simple
OD-characterizable groups $S$ known thus far, the set
$\Omega_0(S)$ is not empty.

We now return to studying the automorphism groups of projective
special linear groups over binary field. It has already been
shown that the automorphism groups: ${\rm Aut}(L_2(2))\cong {\rm
Aut}(\mathbb{S}_3)\cong \mathbb{S}_3$, ${\rm
Aut}(L_3(2))\cong{\rm Aut}(L_2(7))={\rm PGL}(2, 7)$, and ${\rm
Aut}(L_4(2))\cong {\rm Aut}(\mathbb{A}_8)=\mathbb{S}_8$, are
OD-characterizable \cite{amr, mz, zliu}. In Section 3, we will
prove that the automorphism groups ${\rm Aut} (L_{p}(2))$ and
${\rm Aut} (L_{p+1}(2))$, where $2^p-1\geqslant 31$ is a Mersenne
prime, are also OD-characterizable. Combining with the above
results, the following theorem is derived.
\begin{theorem}\label{auto}
Let $2^p-1$ be a Mersenne prime. Then the automorphism groups
${\rm Aut} (L_{p}(2))$ and ${\rm Aut} (L_{p+1}(2))$ are
OD-characterizable.
\end{theorem}

Again we have not found any natural number $n\geqslant 2$ for
which $h_{\rm OD}({\rm Aut}(L_n(2)))\geqslant 2$. Hence, we put
forward the following conjecture.
\begin{coj}
The automorphism groups ${\rm Aut} (L_{n}(2))$ for all integers
$n\geqslant 2$ are OD-characterizable.
\end{coj}

We conclude the introduction with notation to be used in the rest
of this paper. Throughout, by $\mathbb{A}_n$ and $\mathbb{S}_n$,
we denote the alternating and the symmetric groups on $n$ letters,
respectively. We denote by ${\rm Syl}_p(G)$ the set of all Sylow
$p$-subgroups of $G$, where $p\in \pi(G)$. Moreover $G_p$ denotes
a Sylow $p$-subgroup of $G$ for $p\in \pi(G)$. If $H$ is a
subgroup of $G$, then $N_G(H)$ is the normalizer of $H$ in $G$.
Given some positive integer $n$ and some prime $p$, denote by
$n_p$ the $p$-part of $n$, that is the largest power of $p$
diving $n$. We denote by $H:K$ (resp. $H\cdot K$) a split
extension (resp. a non-split extension) of a normal subgroup $H$
by another subgroup $K$. Note that, split extensions are the same
as semi-direct products. All further unexplained notation is
standard and refers to \cite{atlas}, for instance.
\section{Preliminaries}
Given a graph $\Gamma=(V, E)$, a set of vertices $I\subseteq V$
is said to be an independent set of $\Gamma$ if no two vertices
in $I$ are adjacent in $\Gamma$. The independence number of
$\Gamma$, denoted by  $\alpha(\Gamma)$, is the maximum cardinality
of an independent set among all independent sets of $\Gamma$. The
following classical bound holds for every graph $\Gamma$ and is
due to Caro and Wei.
\begin{lm}[\cite{caro}, \cite{wei}]\label{lcw}  Let $\Gamma=(V, E)$ be
a graph with independence number $\alpha(\Gamma)$. Then
$$\alpha(\Gamma)\geqslant \sum_{v\in V}\frac{1}{1+d(v)},$$
where $d(v)$ is the degree of the vertex $v$ in $\Gamma$.
\end{lm}

Given a group $G$, for convenience, we will denote $\alpha({\rm
GK}(G))$ as $t(G)$. Moreover, for a vertex $r\in \pi(G)$, let
$t(r,G)$ denote the maximal number of vertices in independent sets
of ${\rm GK}(G)$ containing $r$.
\begin{theorem}[Theorem 1, \cite{vg}]\label{thm-1111}
Let $G$ be a finite group with $t(G)\geq 3$ and $t(2,G)\geq 2$.
Then the following hold:
\begin{itemize}
\item[{$(1)$}]
There exists a finite non-abelian simple group $S$ such that
$S\leq {\bar G}=G/K\leq {\rm Aut}(S)$ for the maximal normal
solvable subgroup $K$ of $G$.
\item[{$(2)$}] For every independent subset $\rho$ of $\pi(G)$ with
$|\rho|\geq 3$ at most one prime in $\rho$ divides the product
$|K|\cdot |\bar{G}/S|$. In particular, $t(S)\geq t(G)-1$.
\item[{$(3)$}] One of the following holds:

$(3.1)$ every prime $r\in \pi(G)$ non-adjacent to $2$ in ${\rm
GK}(G)$ does not divide the product $|K|\cdot |\bar{G}/S|$; in
particular, $t(2, S)\geq t(2,G)$;

$(3.2)$ there exists a prime $r\in \pi(K)$ non-adjacent to $2$ in
${\rm GK}(G)$; in which case $t(G)=3$, $t(2,G)=2$, and $S\cong
\mathbb{A}_7$ or $L_2(q)$ for some odd $q$.
\end{itemize}
\end{theorem}
\begin{lm}[Lemma 8(1), \cite{Gerech-2}]\label{exp-element}
Let $q>1$ be an integer, $m$ be a natural number, and $p$ be an
odd prime. If $p$ divides $q-1$, then $(q^m-1)_p=m_p\cdot
(q-1)_p.$
\end{lm}

Given positive integers $a\geqslant 2$ and $n$, we say that a
prime $p$ is a primitive prime divisor of $a^n-1$ if $p|a^n-1$
and $p\nmid a^k-1$ for $1\leqslant k<n$. We denote by ${\rm
ppd}(a^n-1)$ the set (depending on $a$ and $n$) of all primitive
prime divisors of $a^n-1$. For example, we have ${\rm
ppd}(13^{11}-1)=\{23, 419, 859, 18041\}$. We recall that, by
Zsigmondy's theorem \cite{zsigmondy} which is given below, the
set ${\rm ppd}(a^n-1)$ is non-empty if $n\neq 2, 6$.
\begin{theorem}[Zsigmondy's Theorem]\label{zsig}
Let $a$, $b$ and $n$ be positive integers such that $(a, b)=1$.
Then there exists a prime $p$ with the following properties:
\begin{itemize}
\item $p$ divides $a^n-b^n$,
\item $p$ does not divide $a^k-b^k$ for all $k<n$,
\end{itemize}
\noindent  with the following exceptions: $a=2$; $b=1$; $n=6$ and
$a+b=2^k$; $n=2$.
\end{theorem}

Primitive prime divisors have been applied in Finite Group Theory
(see \cite{pra, vv}, for example). In fact, the order of any
finite simple group of Lie type $S$ of rank $n$ over a field
${\rm GF}(q)$ is equal to
$$|S|=\frac{1}{d}q^N(q^{m_1}\pm 1)(q^{m_2}\pm 1)\cdots(q^{m_n}\pm 1),$$
(see 9.4.10 and 14.3.1 in \cite{carter}). Therefore any prime
divisor $r$ of $|S|$ distinct from the characteristic $p$ is a
primitive prime divisor of $q^m-1$, for some natural $m$. In
particular, if $S=L_n(q)$, with $q=p^f$, then we have
$$|S|=\frac{1}{(n, q-1)}q^{n\choose 2}(q^2-1)(q^3-1)\cdots (q^{n}-1).$$
Now, it is easy to see that
$$\pi(S)\setminus \{p\}=\bigcup_{i=2}^n {\rm ppd}(q^i-1).$$

The following lemma (which is an immediate corollary of
\cite[Propositions 2.1, 3.1 (1)]{vv}) gives the adjacency
criterion for two prime divisors in the prime graph associated
with a projective special linear groups $L_n(2)$.

\begin{lm}\label{criterion}
Let $L$ be the projective special linear group $L_n(2)$,  with
$n\geqslant 3$. Let $r, s \in \pi(L)\setminus \{2\}$ with $r\in
{\rm ppd}(2^k-1)$ and $s\in {\rm ppd}(2^l-1)$ and assume that
$2\leqslant k\leqslant l$. Then
\begin{itemize}
\item [{$(1)$}]  $r$ and $2$ are adjacent if and only if $k\leqslant n-2$;
\item [{$(2)$}]
$r$ and $s$ are adjacent if and only if $k+l\leqslant n$ or $l$ is
divisible by $k$.
\end{itemize}
In particular, every two prime divisors of $2^m-1$, for a fixed
natural number $m\leqslant n$, are adjacent in ${\rm GK}(L)$.
\end{lm}

The next result which completely determines the degree of all
vertices in the Gruenberg-Kegel graph ${\rm GK}(L_n(2))$, is a
simple consequence of Lemma \ref{criterion}.

\begin{corollary}\label{coro-1}
Let $L$ be the projective special linear group $L_n(2)$ with
$n\geqslant 3$. Let $r\in \pi(L)$ be an odd prime and $r\in {\rm
ppd}(2^k-1)$. Then the following hold:
\begin{itemize}
\item [{$(a)$}] $\deg_L(2)=|\pi(L)|-|{\rm
ppd}(2^{n-1}-1)\cup {\rm ppd}(2^n-1)|-1$. In particular, $t(2,
L)\geq 2$.
\item [{$(b)$}] If $k=n$ or $n-1$, then
$\deg_L(r)=|\pi(2^k-1)|-1$.

\item [{$(c)$}] If $k\neq n, n-1$, then
$$\deg_L(r)=\left\{\begin{array}{lll} |\bigcup\limits_{i=2}^{n-k}{\rm
ppd}(2^i-1)| +|{\rm
ppd}(2^{[\frac{n}{k}]k}-1)| & & k\leqslant n/2,\\[0.5cm]
|\pi(2^k-1)\cup\bigcup\limits_{i=2}^{n-k}{\rm
ppd}(2^i-1)| & & k>n/2.\\
\end{array}\right.
$$ \end{itemize}
\end{corollary}
\begin{proof}
Recall that, the order of $L$ is equal to
$$|L|=2^{n\choose 2}(2^2-1)(2^3-1)\cdots(2^{n-1}-1)(2^n-1).$$
Therefore any odd prime divisor $r$ of $|L|$ is a primitive
divisor of $2^m-1$, for some natural number $m\leqslant n$.

$(a)$ By Lemma \ref{criterion} $(1)$, we have $2\sim r$ if and
only if $k\leqslant n-2$. Therefore, we obtain
$$\deg_L(2)=|\pi(L)|-|{\rm
ppd}(2^{n-1}-1)\cup {\rm ppd}(2^n-1)|-1.$$  In addition, since
$$(2^{n-1}-1, 2^{n}-1)=2^{(n-1, n)}-1=1,$$ and by Theorem
\ref{zsig}, we get $$|{\rm ppd}(2^{n-1}-1)\cup {\rm
ppd}(2^n-1)|\geqslant 2.$$ Therefore, we obtain
$\deg_L(2)\leqslant |\pi(L)|-3$, which forces $t(2,L)\geqslant
2$, as required.

$(b)$ If $k=n$ or $n-1$, then by Lemma \ref{criterion} $(1)$,
$2\nsim r$, and if $s\in \pi(L)\setminus \{2, r\}$ with $s\in
{\rm ppd}(2^l-1)$, then by Lemma \ref{criterion} $(2)$, $s\sim r$
if and only if $l$ divides $k$. But then $2^l-1$ divides $2^k-1$,
and so $s\in \pi(2^k-1)$. Finally, in both cases, we obtain
$\deg_L(r)=|\pi(2^k-1)|-1$.

$(c)$ The conclusion follows immediately from Lemma
\ref{criterion}.
\end{proof}

We are now able to compute the degree pattern of simple group
$L_n(2)$, for a fixed $n$.

\begin{center}
{\bf Table 2.} {\em The degree pattern of some linear groups
$L_n(2)$.}\\[0.2cm] $ \begin{array}{|l|l|}
\hline
L_n(2)  & D(L_n(2))\\
\hline
L_2(2) & (0, 0)\\
L_3(2) & (0, 0, 0)\\
L_4(2) & (1, 2, 1, 0)\\
L_5(2) & (2, 3, 1, 2, 0)\\
L_6(2) & (3, 3, 2, 2, 0)\\
L_7(2) & (4, 4, 3, 3, 2, 0)\\
L_8(2) & (4, 5, 4, 4, 2, 3, 0)\\
L_9(2) & (5, 6, 5, 5, 2, 4, 1, 2)\\
L_{10}(2) & (6, 7, 5, 6, 2, 3, 5, 1, 3)\\
L_{11}(2) & (7, 8, 6, 7, 2, 4, 1, 5, 3, 1, 4)\\
L_{12}(2) & (8, 9, 7, 8, 3, 3, 4, 1, 6, 3, 1, 5)\\
L_{13}(2) & (10, 11, 8, 9, 4, 3, 5, 3, 7, 4, 3, 5, 0)\\
L_{14}(2) & (11, 12, 9, 11, 5, 4, 5, 4, 8, 2, 5, 4, 6, 0)\\
L_{15}(2) & (12, 13, 11, 12, 5, 4, 6, 5, 9, 2, 5, 5, 7, 2, 2)\\
L_{16}(2) & (13, 14, 12, 13, 5, 4, 7, 6, 11, 3, 6, 6, 8, 2, 3, 3)\\
L_{17}(2) & (14, 15, 13, 14, 6, 5, 8, 6, 12, 4, 7, 6, 9, 4, 3, 4, 0)\\
L_{18}(2) & (15, 16, 14, 15, 7, 5, 9, 3, 7, 13, 5, 8, 7, 11, 4, 4, 5,  0)\\
L_{19}(2) & (16, 17, 15, 16, 8, 6, 11, 3, 8, 14, 6, 9, 8, 12, 5, 5, 5, 2, 0)\\
L_{20}(2) & (17, 18, 16, 17, 9, 7, 12, 4, 9, 15, 4, 6, 11, 9, 13, 5, 5, 6, 3, 0)\\
\hline
\end{array}$
\end{center}
It may be finally worth noting that $L_n(2)\hookrightarrow
L_{n+1}(2)$, which implies that:
\begin{itemize}
\item If $n\neq 5$, then $\pi(L_n(2))\subsetneqq \pi(L_{n+1}(2))$ and $\pi_e(L_n(2))\subsetneqq
\pi_e(L_{n+1}(2))$. Moreover, we have
$\pi(L_5(2))=\pi(L_{6}(2))$, while $\pi_e(L_5(2))\subsetneqq
\pi_e(L_{6}(2))$.

\item The Gruenberg-Kegel graph ${\rm GK}(L_n(2))$ is a subgraph of  ${\rm GK}(L_{n+1}(2))$,
\item  If $p\in \pi(L_n(2))$, then $\deg_{L_n(2)}(p)\leqslant
\deg_{L_{n+1}(2)}(p)$.
\end{itemize}

The following lemma (which is taken from \cite[Lemma 8]{silvia})
shows that none of the sets of ``generalized nonnegative matrices"
which we mentioned in Sections 1 and 2 is a convex set.

\begin{lm}[Lemma 8, \cite{silvia}]\label{l-independence} Let $G$ be a
group. If $t(G)\geqslant 3$, then $G$ is non-solvable.
\end{lm}

\begin{lm}\label{r-result} Let $p$ be an odd prime and $L\in \{L_p(2), L_{p+1}(2)\}$. Suppose $G$ is a finite group which
satisfies the conditions $|G|=|L|$ and $D(G)=D(L)$. Then there
hold.
\begin{itemize}
\item [{$(a)$}] There exist three primes in $\pi(G)$ pairwise non-adjacent in ${\rm GK}(G)$,
that is $t(G)\geqslant 3$. In particular, $G$ is a non-solvable
group.

\item [{$(b)$}] There exists an odd prime in $\pi(G)$ which is not adjacent
to the prime $2$ in ${\rm GK}(G)$; that is $t(2,G)\geqslant 2$.

\item [{$(c)$}] There exists a finite non-abelian simple group $S$ such that
$S\leq G/K\leq {\rm Aut}(S)$ for the maximal normal solvable
subgroup $K$ of $G$. Furthermore, $t(S)\geqslant t(G)-1$.
\end{itemize}
\end{lm}
\begin{proof} $(a)$ Suppose first that $L=L_p(2)$.
If $p=3$ (resp. $5$, $7$), then the set $\{2, 3, 5\}$ (resp.
$\{5, 7, 31\}$, $\{7, 31, 127\}$) is an independent set in ${\rm
GK}(G)$, and hence $t(G)\geqslant 3$. Therefore, we may assume
that $p\geqslant 11$.

Assume to the contrary that $t(G)\leqslant 2$. We now point out
some elementary facts about the degree of vertices in ${\rm
GK}(G)$. Firstly, with a similar argument, as in the proof of
Proposition 2.1 in \cite{vv}, we can verify that
$${\rm ppd}(2^p-1)=\pi(2^p-1).$$
Secondly, for two non-adjacent vertices $p_1, p_2\in \pi(G)$,
since $t(G)\leqslant 2$, we obtain
\begin{equation}\label{equation-1}
\deg_G(p_1)+\deg_G(p_2)\geqslant |\pi(G)|-2.
\end{equation}

In what follows, for the sake of convenience, we put $|{\rm
ppd}(2^p-1)|=m$ and $|\pi(G)|=n$. We now consider two cases
separately.

{\em Case $1$.} $m=1$. Suppose that $\pi(2^p-1)=\{q\}$. Then
$\deg_G(q)=0$, and so the Gruenberg-Kegel graph ${\rm GK}(G)$ is
not connected. On the other hand, by Corollary \ref{coro-1} $(a)$
and Theorem \ref{zsig}, we obtain
$$\deg_G(2)=n-|{\rm ppd}(2^{p-1}-1)|-2\leqslant n-3.$$ Hence,
there exists a prime $q'\in \pi(G)\setminus \{q\}$ such that
$q'\nsim 2$. Therefore, the set $\{2, q', q\}$ is an independent
set in ${\rm GK}(G)$, against our hypothesis.

{\em Case $2$.} $m\geqslant 2$. Suppose that ${\rm
ppd}(2^p-1)=\{p_1, p_2, \ldots, p_m\}$. If there exists $p_i$
such that $2\nsim p_i$, then, from Eq. (\ref{equation-1}), we
conclude that
$$\deg_G(2)+\deg_G(p_i)\geqslant n-2.$$
Applying Corollary \ref{coro-1} $(a), (b)$ and some
simplification this leads to
$$n-|\pi(2^{p-1}-1)|\geqslant n,$$
which is a contradiction. Therefore, we may assume that $2\sim
p_i$, for each $i=1, 2, \ldots, m$, and so $\deg_G(2)\geqslant
m$. Now we apply Corollary \ref{coro-1} $(a)$, to get
$$n-m-1>n-m-|{\rm ppd}(2^{p-1}-1)|-1=\deg _G(2)\geqslant m,$$
or equivalently, $m<\frac{n-1}{2}$. Furthermore, there are two
primes $p_i, p_j$ such that $p_i\nsim p_j$ in ${\rm GK}(G)$,
otherwise $\deg_G(p_i)\geqslant m$ and this contradicts the fact
that $\deg_G(p_i)=m-1$. Again, by Eq. (\ref{equation-1}),
$$ \deg_G(p_i)+\deg_G(p_j)\geqslant n-2,$$
which forces $m\geqslant \frac{n}{2}$, a contradiction. This
completes the proof for $L=L_p(2)$.

In the case when $L=L_{p+1}(2)$, the proof is similar to the
previous case and, therefore, omitted. Finally, in both cases,
$t(G)\geqslant 3$ and by Lemma \ref{l-independence}, $G$ is a
non-solvable group.

$(b)$ It is obvious, because $\deg_G(2)\leqslant |\pi(G)|-3$.

$(c)$ Follows from $(a)$, $(b)$ and Theorem \ref{thm-1111}.
\end{proof}
\begin{proposition}[Theorem A, \cite{w}]\label{GK-W} If $G$ is a finite group
with disconnected Gruenberg-Kegel graph ${\rm GK}(G)$, then one of
the following holds:
\begin{itemize}
\item[{\rm (a)}] $s(G)=2$, $G$ is a Frobenius group.

\item[{\rm (b)}] $s(G)=2$, $G=ABC$, where $A$ and $AB$ are normal subgroups of $G$, $B$ is a
normal subgroup of $BC$, and $AB$ and $BC$ are Frobenius groups
(such a group $G$ is called a $2$-Frobenius group).

\item [{\rm (c)}] There exists a non-abelian simple group $P$ such that
$P\leqslant G/K\leqslant {\rm Aut}(P)$ for some nilpotent normal
$\pi_1(G)$-subgroup $K$ of $G$, and $G/P$ is a $\pi_1(G)$-group.
Moreover, ${\rm GK}(P)$ is disconnected, $s(P)\geqslant s(G)$.
\end{itemize}
\end{proposition}

The following Propositions deal with the structure of Frobenius
and $2$-Frobenius groups and their Gruenberg-Kegel graphs. One may
find their proofs in \cite{g, mazurov}.
\begin{proposition}[Theorem 3.1, \cite{g}] \label{frob}  If $G$ is a Frobenius
group with the kernel $K$ and complement $C$, then the following
conditions hold:
\begin{itemize}
\item[{\rm (1)}]
$K$ is nilpotent and so its Gruenberg-Kegel graph ${\rm GK}(K)$
is a complete graph, that is ${\rm GK}(K)=K_{|\pi(K)|}$;
\item[{\rm (2)}]  $s(G)=2$ and the connected components of
${\rm GK}(G)$ are ${\rm GK}(K)$ and ${\rm GK}(C)$, that is, ${\rm
GK}(G)={\rm GK}(K)\oplus {\rm GK}(C)$. In particular, we have
${\rm OC}(G)=\{|K|, |C|\}$.
\item[{\rm (3)}] $|C|$ divides $|K|-1$, and so $|C|<|K|$.
\end{itemize}
\end{proposition}

\begin{proposition}[Lemma 7, \cite{mazurov}]\label{2frob}  In case {\rm (b)} of Proposition $\ref{GK-W}$:
\begin{itemize}
\item[{\rm (1)}] $C$ and $B$ are cyclic groups, and $|B|$ is odd;

\item[{\rm (2)}] ${\rm GK}(B)$ and ${\rm GK}(AC)$ are connected components of
the prime graph ${\rm GK}(G)$, and both of them are complete
graphs. Hence, we have  $${\rm GK}(G)={\rm GK}(AC)\oplus {\rm
GK}(B)=K_{|\pi(AC)|}\oplus K_{|\pi(B)|}.$$ In particular,
$s(G)=2$, $\pi_1(G)=\pi(AC), \pi_2(G)=\pi(B)$, ${\rm
OC}(G)=\{|AC|, |B|\}$, and for every primes $p\in \pi(G)$, we have
$\deg_G(p)=|\pi(AC)|-1$ or $|\pi(B)|-1$.
\end{itemize}
\end{proposition}

The following result will be used frequently throughout next
section.
\begin{lm}\label{inK} Let $G$ be a finite group and $K$ be a normal solvable subgroup
of $G$. Let $p, q\in \pi(G)$ such that $p\not\equiv 1\pmod{q}$,
$q\not\equiv 1\pmod{p}$ and $|G_pG_q|=pq$. If $p$ divides the
order of $K$, then $p\sim q$ in ${\rm GK}(G)$.
\end{lm}
\begin{proof} If $q\in \pi(K)$, then $K$ contains a cyclic
subgroup of order $pq$, and the result is proved. Hence, we may
assume that $q\notin \pi(K)$. Let $P$ be a Sylow $p$-subgroup of
$K$. Then $G=KN_G(P)$ by Frattini argument, and so $N_G(P)$
contains an element of order $q$, say $x$. Clearly $P\langle x
\rangle$ is a cyclic subgroup of $G$ of order $pq$, and hence
$p\sim q$ in ${\rm GK}(G)$. This completes the proof.
\end{proof}

We now present the following useful degree criterion for
non-solvability a group using whose degree pattern.
\begin{lm}\label{non-solvablility}
Let $G$ be a finite group satisfies $\Omega_0(G)\neq \emptyset$
and $\Omega_i(G)\neq \emptyset$ for some $1\leqslant i\leqslant
|\pi(G)|-3$ (i.e., there exists a vertex in ${\rm GK}(G)$ of
degree at most $|\pi(G)|-3$). Then $t(G)\geqslant 3$ and
especially $G$ is non-solvable.
\end{lm}

We omit the straightforward proof.

\begin{lm}\label{43} (\cite{banoo-alireza}, \cite{za})
Let $S$ be a finite non-abelian simple group such that its order
divides $|L_{n}(2)|$ where $n\in \{10, 11\}$. Then
\begin{itemize}
\item [{$(1)$}] if $n=10$ and $\{11, 73\}\subset \pi(S)$, then  $S$ is isomorphic to $L_{10}(2)$
\item [{$(2)$}] If $n=11$ and $\{23, 89\}\subset \pi(S)$, then  $S$ is isomorphic to
$L_{11}(2)$.
\end{itemize}
\end{lm}
\begin{proof}
By results collected in \cite{banoo-alireza, za}, if $S$ is a
finite non-abelian simple group such that its order divides the
order of $L_{11}(2)$, then $S$ is isomorphic to one of the simple
groups listed below in Table 3. Now, the lemma follows by
checking the conditions in $(1)$ and $(2)$.
\end{proof}
\begin{center}
{\small {\bf Table 3}. {\em The simple group $S$ whose order
divides $|L_{11}(2)|=2^{55}\cdot 3^6\cdot 5^2 \cdot 7^3\cdot
11\cdot 17 \cdot 23\cdot
31^2 \cdot 73\cdot 89 \cdot 127$.}}\\[0.4cm]
\begin{tabular}{ll}
$\begin{tabular}{|l|l|} \hline
&\\[-0.26cm]
$S$ & $|S|$ \\[0.1cm]
\hline
&\\[-0.26cm]
$\mathbb{A}_5$ & $2^2\cdot3\cdot5$  \\ [0.1cm] $\mathbb{A}_6$ &
$2^3\cdot3^2\cdot5$ \\ [0.1cm]
$U_4(2)$& $2^6\cdot 3^4\cdot 5$  \\[0.1cm]
$\mathbb{A}_7$ & $2^3\cdot3^2\cdot5\cdot7$  \\ [0.1cm]
$\mathbb{A}_8$ & $2^6\cdot3^2\cdot5\cdot7$  \\ [0.1cm]
$\mathbb{A}_9$ & $2^6\cdot3^4\cdot5\cdot7$  \\ [0.1cm]
$\mathbb{A}_{10}$ & $2^7\cdot3^4\cdot5^2\cdot7$  \\ [0.1cm]
$B_3(2)$ & $2^9\cdot 3^4\cdot 5\cdot 7$ \\[0.1cm]
$O_8^+(2)$ & $2^{12}\cdot 3^5\cdot 5^2\cdot 7$ \\[0.1cm]
$L_3(2^2)$ & $2^6\cdot 3^2\cdot 5\cdot 7$ \\[0.1cm]
$L_2(2^3)$ & $2^3\cdot 3^2\cdot 7$  \\[0.1cm]
$U_3(3)$ & $2^5\cdot 3^3\cdot 7$  \\[0.1cm]
$U_4(3)$ & $2^7\cdot 3^6\cdot 5\cdot 7$ \\[0.1cm]
$L_2(7)$& $2^3\cdot 3\cdot 7$  \\[0.1cm]
$L_2(7^2)$ & $2^4\cdot 3\cdot 5^2\cdot 7^2$  \\[0.1cm]
$J_2$ & $2^7\cdot 3^3\cdot 5^2\cdot 7$ \\[0.1cm]
$U_5(2)$ & $2^{10}\cdot3^5\cdot5\cdot11$  \\ [0.1cm] $U_6(2)$ &
$2^{15}\cdot3^6\cdot5\cdot7\cdot11$  \\ [0.1cm]
$L_2(11)$ & $2^2\cdot3\cdot5\cdot11$  \\[0.1cm]
$M_{11}$ & $2^4\cdot3^2\cdot5\cdot11$  \\ [0.1cm] $M_{12}$ &
$2^6\cdot3^3\cdot5\cdot11$ \\ [0.1cm] $M_{22}$ &
$2^7\cdot3^2\cdot5\cdot7\cdot11$  \\[0.1cm]
$\mathbb{A}_{11}$ &
$2^7\cdot3^4\cdot5^2\cdot7\cdot11$  \\[0.1cm]
$\mathbb{A}_{12}$ &
$2^9\cdot3^5\cdot5^2\cdot7\cdot11$  \\[0.1cm]
$C_4(2)$&$2^{16}\cdot 3^5\cdot 5^2\cdot 7\cdot 17$\\[0.1cm]
\hline
\end{tabular} $ & \!\!\!\!\!
$\begin{tabular}{|l|l|} \hline
&\\[-0.26cm]
$S$ & $|S|$ \\[0.1cm]
\hline
&\\[-0.26cm]
$O_{8}^-(2)$&$2^{12}\cdot 3^4\cdot 5\cdot 7\cdot 17$\\[0.1cm]
$O_{10}^-(2)$& $2^{20}\cdot 3^6\cdot 5^2\cdot 7\cdot 11\cdot17$\\[0.1cm]
$L_4(2^2)$&$2^{12}\cdot 3^4\cdot 5^2\cdot 7\cdot 17$\\[0.1cm]
$C_2(2^2)$&$2^8\cdot 3^2\cdot5^2\cdot 17$\\[0.1cm]
$L_2(2^4)$&$2^4\cdot 3\cdot 5\cdot 17$\\[0.1cm]
$L_2(17)$&$2^4\cdot 3^2\cdot 17$\\[0.1cm]
${\rm He}$&$2^{10}\cdot 3^3\cdot 5^2\cdot 7^3\cdot 17$\\[0.1cm]
$L_{2}(23)$&$2^{3}\cdot 3\cdot 11\cdot 23$\\[0.1cm]
$M_{23}$ & $2^{7}\cdot 3^2\cdot 5\cdot 7\cdot 11\cdot 23$\\[0.1cm]
$M_{24}$ & $2^{10}\cdot 3^3\cdot 5\cdot 7\cdot 11\cdot 23$\\[0.1cm]
$L_2(31)$&$2^5\cdot 3\cdot 5\cdot 31$\\[0.1cm]
$L_5(2)$&$2^{10}\cdot 3^2\cdot 5\cdot 7\cdot 31$\\[0.1cm]
$L_6(2)$&$2^{15}\cdot 3^4\cdot 5\cdot 7^2\cdot 31$\\[0.1cm]
$S_{10}(2)$&$2^{25}\cdot 3^6\cdot 5^2\cdot 7\cdot 11\cdot 17\cdot
31$\\[0.1cm]
$O_{10}^+(2)$&$2^{20}\cdot 3^5\cdot 5^2\cdot 7\cdot 17\cdot31$\\[0.1cm]
$L_5(2^2)$&$2^{20}\cdot 3^5\cdot 5^2\cdot 7\cdot 11\cdot 17\cdot31$\\[0.1cm]
$L_2(2^5)$&$2^5\cdot 3\cdot 11\cdot 31$\\[0.1cm]
$L_3(2^3)$ & $2^9\cdot3^2\cdot7^2\cdot73$  \\[0.1cm]
$U_3(3^2)$ & $2^5\cdot3^6\cdot5^2\cdot73$ \\[0.1cm]
$L_2(89)$ & $2^{3}\cdot 3^2\cdot 5 \cdot 11\cdot 89$  \\[0.1cm]
$L_7(2)$ & $2^{21}\cdot 3^4\cdot 5 \cdot 7^2\cdot 31 \cdot 127$ \\[0.1cm]
$L_8(2)$ & $2^{28}\cdot 3^5\cdot 5^2 \cdot 7^2\cdot 17 \cdot 31 \cdot 127$  \\[0.1cm]
$L_9(2)$ & $2^{36}\cdot 3^5\cdot 5^2 \cdot 7^3\cdot 17 \cdot
31\cdot 73 \cdot 127$  \\[0.1cm]
$L_{10}(2)$ & $2^{45}\cdot 3^6\cdot 5^2 \cdot 7^3\cdot 11\cdot 17
\cdot
31^2\cdot 73 \cdot 127$  \\[0.1cm]
$L_{11}(2)$ & $2^{55}\cdot 3^6\cdot 5^2 \cdot 7^3\cdot 11\cdot 17
\cdot 23\cdot 31^2 \cdot 73\cdot 89 \cdot 127$  \\[0.1cm]
\hline
\end{tabular} $
\end{tabular}
\end{center}
\section{OD-Characterizability of Certain Groups}
As we mentioned earlier in the Introduction, we are going to show
that the simple groups $L_{10}(2)$, $L_{11}(2)$ and the
automorphism groups ${\rm Aut} (L_{p}(2))$ and ${\rm Aut}
(L_{p+1}(2))$, where $2^p-1$ is a Mersenne prime, are uniquely
determined through their orders and degree patterns.
\subsection{OD-Characterizability of Simple Groups $L_{10}(2)$ and
$L_{11}(2)$}Here, we show that the simple groups $L_{10}(2)$ and
$L_{11}(2)$ are OD-characterizable. We start with the following
theorem.
\begin{theorem}\label{n10} Let $G$ be a finite group which
satisfies the following conditions:
\begin{itemize}
\item $|G|=2^{45}\cdot 3^6\cdot 5^2 \cdot 7^3\cdot 11\cdot 17 \cdot
31^2 \cdot 73 \cdot 127$,  and \item $D(G)=(6, 7, 5, 6, 2, 3, 5,
1, 3)$.
\end{itemize}
Then $G$ is isomorphic to $L_{10}(2)$.
\end{theorem}
\begin{proof}
Applying Lemma \ref{lcw} and easy computations show that
$$t(G)\geqslant \sum_{p\in \pi(G)}\frac{1}{1+\deg_G(p)}\approx 2.07.$$
Hence, $t(G)\geqslant 3$ and $G$ is a non-solvable group by Lemma
\ref{l-independence}. In addition, since
$\deg_G(2)=|\pi(G)|-3=6$, $t(2,G)\geqslant 2$. Let $K$ be the
maximal normal solvable subgroup of $G$. Then, by Theorem
\ref{thm-1111}, there exists a finite non-abelian simple group
$S$ such that $S\leq G/K\leq {\rm Aut}(S)$. Evidently, $K$ is a
$\{11, 73\}'$-group, since otherwise by Lemma \ref{inK}, we
obtain $\deg_G(11)\geqslant 3$ or $\deg_G(73)\geqslant 3$, which
is a contradiction. Now, it is clear that $|S|$ is divisible by
11 and 73, and from Lemma \ref{43} $(1)$, it follows that $S\cong
L_{10}(2)$. Finally, since $|G|=|L_{10}(2)|$, we conclude that
$|K|=1$ and $G\cong L_{10}(2)$. \end{proof}

\begin{theorem}\label{n11} Let $G$ be a finite group. Then $G\cong {L_{11}(2)}$
if and only if $G$ satisfies the following conditions:
\begin{itemize}
\item $|G|=2^{55}\cdot 3^6\cdot 5^2 \cdot 7^3\cdot 11\cdot 17
\cdot23\cdot  31^2 \cdot 73 \cdot 89\cdot 127$,  and \item
$D(G)=(7, 8, 6, 7, 2, 4, 1, 5, 3, 1, 4)$.
\end{itemize}
\end{theorem}
\begin{proof} First of all, it follows from Lemma \ref{r-result} $(a)$, $t(G)\geqslant 3$ and $G$ is a
non-solvable group. Moreover, since $\deg_G(2)=|\pi(G)|-4=7$,
$t(2,G)\geqslant 2$. Let $K$ be the maximal normal solvable
subgroup of $G$. Then, by Theorem \ref{thm-1111}, there exists a
finite non-abelian simple group $S$ such that $S\leq G/K\leq {\rm
Aut}(S)$. Moreover, one can easily see that $K$ is a $\{23,
89\}'$-group, which follows directly from Lemma \ref{inK} and the
facts that $\deg(23)=\deg(89)=1$. Now, it is clear that $|S|$ is
divisible by $23$ and $89$. Using Lemma \ref{43} $(2)$, it
follows that $S\cong L_{11}(2)$. Finally, since
$L_{11}(2)\leqslant G/K \leqslant {\rm Aut}(L_{11}(2))$ and
$|G|=|L_{11}(2)|$, we deduce that $|K|=1$ and $G\cong L_{11}(2)$.
\end{proof}


\subsection{On the Automorphism Group of $L_n(2)$}
It is known (see \cite[Theorem 2.5.12]{GorLS}) that, the group of
outer automorphisms of a simple group of Lie type is generated by
the diagonal automorphisms, the graph automorphisms (of the
underlying Dynkin diagram), and the field automorphisms of the
field of definition. Especially, for $S=L_n(q)$, with $n\geqslant
2$ and $q=p^f$, we have (see also \cite{atlas}):
$$|{\rm Out}(S)|=(n, q-1)\cdot f\cdot 2.$$
Therefore, the only outer automorphism of simple group $L_n(2)$,
$n\geq 3$, is the graph automorphism of order $2$, corresponds to
the symmetry of its Dynkin diagram. We denote by $\sigma$ this
automorphism and set $L:=L_n(2)$. Then, we have ${\rm
Aut}(L)=L\cdot \langle \sigma \rangle$, and so $|{\rm
Aut}(L):L|=2$. The following general results may be stated:

\begin{lm}\label{aut-1}
Let $S$ be a simple group with $|{\rm Aut}(S):S|=2$. Then there
holds: $${\rm GK}({\rm Aut}(S))-\{2\}={\rm GK}(S)-\{2\}.$$ In
particular, if $r\in \pi(S)-\{2\}$, then $\deg_S(r)\leqslant
\deg_{{\rm Aut}(S)}(r) \leqslant \deg_S(r)+1,$ and in addition,
if $2\sim r$ in ${\rm GK}(S)$, then $\deg_{{\rm
Aut}(S)}(r)=\deg_S(r)$.
\end{lm}
\begin{proof} First of all, we note that $S\cong {\rm Inn}(S)\leqslant {\rm Aut}(S)$, and
so $\pi_e(S)\subseteq \pi_e({\rm Aut}(S))$ and ${\rm GK}(S)$ is a
subgraph of ${\rm GK}({\rm Aut}(S))$. We claim that $\pi_e({\rm
Aut}(S))\setminus \pi_e(S)$ is a subset of the set of even natural
numbers. Suppose $m\in \pi_e({\rm Aut}(S))\setminus \pi_e(S)$ is
an odd number. Then there exists $x\in {\rm Aut}(S)\setminus S$
such that $o(x)=m$. On the other hand, we have $x^{-1}=x^{m-1}
\in S$, since $|{\rm Aut}(S):S|=2$ and $m-1$ is even. Hence $x\in
S$ , which is a contradiction.

Notice that $\pi({\rm Aut}(S))=\pi(S)$. In what follows we claim
that if $p$ and $q$ are two odd primes such that $p\nsim q$ in
${\rm GK}(S)$, then $p\nsim q$ in ${\rm GK}({\rm Aut}(S))$. Assume
that the claim is false and $p\sim q$ in ${\rm GK}({\rm Aut}(S))$.
Then $S$ dose not contain an element of order $pq$, while from the
previous paragraph of the proof the automorphism group ${\rm
Aut}(S)$ has an element of order $2pq$, say $x$. Therefore $x^2\in
S$ and $o(x^2)=pq$, which is a contradiction.
\end{proof}

A sequence of non-negative integers $(a_1, a_2, \ldots, a_k)$ is
said to be {\em majorised} by another such sequence $(b_1, b_2,
\ldots, b_k)$ if $a_i\leqslant b_i$ for $1\leqslant i\leqslant k$.
A graph $\Gamma_1$ is {\em degree-majorised} by a graph
$\Gamma_2$ if $V(\Gamma_1)=V(\Gamma_2)$ and the {\em non-ascending
degree sequence} of $\Gamma_1$ is majorised by that of $\Gamma_2$.
By Lemma \ref{aut-1}, we have immediately the following:
\begin{corollary}\label{majorised-1}
Let $S$ be a simple group with $|{\rm Aut}(S):S|=2$. Then ${\rm
GK}(S)$ is degree-majorised by ${\rm GK}({\rm Aut}(S))$.
\end{corollary}

Hereinafter, we assume that $L:=L_n(2)$ with $n\in \{p, p+1\}$,
where $p$ is an odd prime. We list some elementary properties of
the automorphism group ${\rm Aut}(L)$ that are useful in the
following:
\begin{itemize}
\item $|{\rm Aut}(L)|=2\cdot |L|=2^{{n\choose
2}+1}(2^2-1)(2^3-1)\cdots(2^n-1)$ and $\pi({\rm Aut}(L))=\pi(L)$.
\item $s({\rm Aut}(L))=2$ (see \cite[Lemma 2.2]{d-m}).
\item $\pi_1({\rm Aut}(L))=\pi_1(L)$ and $\pi_2({\rm
Aut}(L))=\pi_2(L)=\pi(2^p-1)$. In fact, if $n=p$, then
$$C_L(\sigma)\cong PSO^+(p, 2) \ \mbox{ of order} \ \
2^{((p-1)/2)^2}(2^2-1)(2^4-1)\cdots(2^{p-1}-1),$$ and if $n=p+1$,
then
$$C_L(\sigma)\cong PSp(p+1, 2) \ \mbox{ of order} \ \
2^{((p+1)/2)^2}(2^2-1)(2^4-1)\cdots(2^{p-1}-1)(2^{p+1}-1),$$ (see
\cite[19.9]{ash-seitz}). Therefore, if $q\in {\rm ppd}(2^p-1)$,
then $q\nsim 2$ in ${\rm GK}({\rm Aut}(L))$. Moreover, by Lemma
\ref{aut-1} and the fact that $\pi_2(L)=\pi(2^p-1)$, $q$ is not
adjacent to any odd primes in $\pi_1(L)\setminus \pi(2^p-1)$.
\end{itemize}

In the sequel, we will show that the automorphism group of linear
groups $L_{p}(2)$ and $L_{p+1}(2)$, where $2^p-1$ is a Mersenne
prime, are uniquely determined through their orders and degree
patterns. We start with the following lemmas.
\begin{lm}\label{m} Let $n\geqslant 3$ be an integer and
$L=L_n(2)$.Then there hold.
\begin{itemize}
\item[{$(1)$}] If $n\geqslant 12$ is even, then $(2^k-1)^2$, $2\leqslant k\leqslant n$, does not divide
the order of ${\rm Aut}(L)$ if and only if $k=\frac{n}{2}+i$,
$i=1, 2, \ldots, \frac{n}{2}$.
\item[{$(2)$}] If $n\geqslant 13$ is odd, then $(2^k-1)^2$, $2\leqslant k\leqslant n$, does not divide
the order of ${\rm Aut}(L)$ if and only if $k=\frac{n-1}{2}+i$,
$i=1, 2, \ldots, \frac{n+1}{2}$.
\item[{$(3)$}] If $n\leqslant 11$, then $(2^{k}-1)^2$ does not divide
the order of ${\rm Aut}(L)$ if and only if one of the following
statements holds:

\subitem {$(3.1)$} $n=11$ and $k=7, 8, 9, 10, 11$.

\subitem {$(3.2)$} $n=10$ and $k=7, 8, 9, 10$.

\subitem {$(3.3)$} $n=9$ and $k=5, 7, 8, 9$.

\subitem {$(3.4)$} $n=8$ and $k=5, 7, 8$.

\subitem {$(3.5)$} $n=7$ and $k=4, 5, 7$.

\subitem {$(3.6)$} $n=6$ and $k=4, 5$.

\subitem {$(3.7)$} $n=5$ and $k=3, 4, 5$.

\subitem {$(3.8)$} $n=4$ and $k=3, 4$.

\subitem {$(3.9)$} $n=3$ and $k=2, 3$.
\end{itemize}
\end{lm}
\begin{proof}
Since the proofs of $(1)$ and $(2)$ are similar, only the proof
for $(1)$ is presented. The proof of $(3)$ is a straightforward
verification. First of all, we recall that  \[  |{\rm
Aut}(L)|=2\cdot |L|=2^{{n\choose 2}+1}\prod_{i=2}^{n}(2^i-1),
\] because $|{\rm Out}(L)|=2$. Moreover, if $s\in {\rm
ppd}(2^k-1)$, then $s|2^l-1$ if and only if $k$ divides $l$ (see
the proof of Proposition 2.1 in \cite{vv}). Assume first that
$k\geqslant \frac{n}{2}+1\geqslant 7$. Applying Theorem
\ref{zsig}, we can consider a primitive prime divisor $s\in {\rm
ppd}(2^{k}-1)$, and suppose that $s^m\parallel 2^{k}-1$. As we
mentioned before, if $s|2^l-1$, then $k$ divides $l$, and hence
$l\geqslant 2k\geqslant n+2$, which means that $(s, |{\rm
Aut}(L)|/(2^{k}-1))=1$, and so $s^m\parallel |{\rm Aut}(L)|$.
Hence, if $(2^{k}-1)^2$ divides $|{\rm Aut}(L)|$ then we must
have $s^{2m}\big | |{\rm Aut}(L)|$, which is a contradiction.
Assume next that $k\leqslant \frac{n}{2}$. In this case
$2k\leqslant n$, and since $2^k-1|2^{2k}-1$, it follows that
$(2^k-1)^2\big | |{\rm Aut}(L)|$. This completes the proof of
$(1)$.
\end{proof}

\begin{lm}\label{independent-3} Let $2^p-1\geqslant 31$ be a
Mersenne prime and $L\in \{L_p(2), L_{p+1}(2)\}$. Suppose that
$G$ is a finite group satisfies the conditions: $|G|=|{\rm
Aut}(L)|$  and $D(G)=D({\rm Aut}(L))$. Then $t(G)\geqslant 3$. In
particular, $G$ is a non-solvable group.
\end{lm}
\begin{proof}
We recall that $\deg_G(2^p-1)=0$, and so $\Omega_0(G)\neq
\emptyset$. To complete the proof, from Lemma
\ref{non-solvablility}, it is enough to show that $\Omega_i(G)\neq
\emptyset$ for some $1\leqslant i\leqslant |\pi(G)|-3$. If $p=5$
(resp. $7$), then $\deg_{L_5(2)}(5)=1$ and $\deg_{L_6(2)}(5)=2$
(resp. $\deg_{L_7(2)}(31)=2$ and $\deg_{L_8(2)}(17)=2$). Hence,
by Lemma \ref{aut-1}, we have
$$\begin{array}{ll} L=L_5(2) & \deg_G(5)=\deg_{{\rm Aut}(L)}(5)\leqslant \deg_L(5)+1=2=|\pi(G)|-3,\\[0.2cm]
L=L_6(2) & \deg_G(5)=\deg_{{\rm Aut}(L)}(5)=\deg_L(5)=2=|\pi(G)|-3,\\[0.2cm]
&  \ \  \mbox{(note that $2\sim 5$ in ${\rm GK}(L)$),}\\
L=L_7(2) & \deg_G(31)=\deg_{{\rm Aut}(L)}(31)\leqslant  \deg_L(31)+1=3=|\pi(G)|-3,\\[0.2cm]
L=L_8(2) & \deg_G(17)=\deg_{{\rm Aut}(L)}(17)\leqslant
\deg_L(17)+1=3=|\pi(G)|-3,
\end{array} $$ as required.

Therefore, we may assume that $p\geqslant 13$. In this case, we
consider a primitive prime divisor of $2^{p-1}-1$, say $r$. By
Lemma \ref{criterion}, one can easily see that $r\nsim s$ in
${\rm GK}(L)$, for each $$s\in \bigcup _{i=1}^{\frac{p-3}{2}}{\rm
ppd}(2^{\frac{p-1}{2}+i}-1).$$ Hence, by Theorem \ref{zsig}, we
obtain $$\deg_L(r)\leqslant |\pi(L)|-|\bigcup
_{i=1}^{\frac{p-3}{2}}{\rm ppd}(2^{\frac{p-1}{2}+i}-1)|-1\leqslant
|\pi(L)|-\frac{p-3}{2}-1\leqslant |\pi(L)|-6,
$$ because $p\geqslant 13$. Finally, we conclude that
$$\deg_G(r)=\deg_{{\rm Aut}(L)}(r)\leqslant
\deg_L(r)+1\leqslant |\pi(L)|-5=|\pi(G)|-5,
$$ as required.
\end{proof}

We are now ready to prove our main result.

\begin{theorem}\label{aut-3} Let $2^p-1$ be a Mersenne prime and $L\in \{L_p(2), L_{p+1}(2)\}$. Suppose that
$G$ is a finite group satisfies the conditions: $|G|=|{\rm
Aut}(L)|$  and $D(G)=D({\rm Aut}(L))$. Then $G$ is isomorphic to
${\rm Aut}(L)$.
\end{theorem}
\begin{proof} First of all, we consider the cases $p=2$ and $3$.
Indeed, if $L=L_2(2)\cong {\Bbb S}_3$, then ${\rm Aut}(L)\cong L$
and the result now follows by applying Theorem $1.2$ in
\cite{amr}. If $L=L_3(2)$, then ${\rm Aut}(L)={\rm PGL}(2, 7)$
and the result is proved  in \cite{zliu}. Finally, if
$L=L_4(2)\cong {\Bbb A}_8$, then ${\rm Aut}(L)\cong {\Bbb S}_8$
and the result follows from Theorem $1.5$ in \cite{mz}.

Therefore, we assume that $2^p-1$ is a Mersenne prime with
$p\geqslant 5$ and $L\in \{L_p(2), L_{p+1}(2)\}$. Let $G$ be a
finite group with $|G|=|{\rm Aut}(L)|$ and $D(G)=D({\rm
Aut}(L))$. Then $\pi(G)=\pi({\rm Aut}(L))=\pi(L)$, $2^p-1$ is the
largest prime in $\pi(G)$ and $\deg_G(2^p-1)=0$. Moreover, by
Corollary \ref{coro-1} (c), $\deg_G(3)=|\pi(G)|-2$, which forces
$s(G)=2$. More precisely, we have
$$\pi_1(G)=\pi_1(L) \ \ \mbox{and} \ \
\pi_2(G)=\{2^p-1\},$$ and since $G$ and ${\rm Aut}(L)$ have the
same order, we conclude that $${\rm OC}(G)={\rm OC}({\rm
Aut}(L))=\{m_1, m_2\},$$ where
$$m_1=\left\{\begin{array}{lllll} 2^{{p\choose
2}+1}(2^2-1)(2^3-1)\cdots(2^{p-1}-1) & & & \mbox{if} & L=L_p(2), \\[0.3cm]
2^{{p+1\choose 2}+1}(2^2-1)(2^3-1)\cdots(2^{p-1}-1)(2^{p+1}-1) & &
& \mbox{if} & L=L_{p+1}(2);
\end{array} \right.$$ and $$m_2=2^p-1.$$
Furthermore, by Proposition \ref{GK-W}, one of the following
cases holds:
\begin{itemize}
\item[{\em Case $1$.}] $G$ is either a Frobenius group or a $2$-Frobenius group;
\item [{\em Case $2$.}] There exists a non-abelian simple group $P$
such that $P\leqslant G/K\leqslant {\rm Aut}(P)$ for some
nilpotent normal $\pi_1(G)$-subgroup $K$ of $G$, and $G/P$ is a
$\pi_1(G)$-group. Moreover, $s(P)\geqslant 2$ and
$\pi_2(G)=\{2^p-1\}$.
\end{itemize}
In what follows, we will consider every case separately.
\begin{lm} Case $1$ is impossible.
\end{lm}
\begin{proof} First of all, by Lemma \ref{independent-3}, $G$ is a non-solvable group.
Hence, $G$ is not a $2$-Frobenius group. Assume now that $G$ is a
Frobenius group with kernel $K$ and complement $C$. Then by
Proposition \ref{frob}, ${\rm OC}(G)=\{|K|, |C|\}$. From $|C|<|K|$
we can easily conclude that $|K|=m_1$ and $|C|=m_2=2^p-1$. But
then, the degree pattern of $G$ has the following form:
$$D(G)=(n-2, n-2, \ldots, n-2, 0),$$
where $n=|\pi(G)|$, and hence $t(G)=2$, which contradicts Lemma
\ref{independent-3}.
\end{proof}

Thus Case 2 holds, that is, there exists a non-abelian simple
group $P$ such that $$P\leqslant G/K\leqslant {\rm Aut}(P),$$ for
some nilpotent normal $\pi_1(G)$-subgroup $K$ of $G$, and $G/P$
is a $\pi_1(G)$-group. Evidently $\pi_2(P)=\{2^p-1\}$ and
$\pi_e(P)\subseteq \pi_e(G/K)\subseteq \pi_e(G)$. Therefore, for
every prime $r\in \pi(P)$, we have $\deg_P(r)\leqslant \deg_G(r)$.

\begin{lm}\label{isomorL} $P$ is isomorphic to $L$.
\end{lm}
\begin{proof} According to the classification of the finite simple groups we
know that the possibilities for $P$ are: alternating groups
$\mathbb{A}_m$, $m\geqslant 5$; 26 sporadic finite simple groups;
simple groups of Lie type. We deal with the above cases
separately. We will use the results summarized in Tables 1, 2 and
3 in \cite{km}.

First, suppose $P$ is an alternating group $\mathbb{A}_m$,
$m\geqslant 5$. Since $2^p-1\in \pi(P)$, $m\geqslant 2^p-1$. Now,
we consider a prime $r$ between $2^{p-1}-1$ and $2^p-1$. It is
clear that $ r\in\pi(\mathbb{A}_m)\backslash \pi(G)$, but this is
impossible.

Next, suppose $P$ is a sporadic simple group. Since the odd order
components of a sporadic simple group are prime less than 71, it
follows that $2^p-1<71$. Hence we obtain that $p=3$ or 5. Using
the results summarized in Tables 1, 2 and 3 in \cite{km}, we see
that $P$ cannot isomorphic to any sporadic simple group.

Finally, suppose $P$ is a simple group of Lie type. Here,
according to the number of the prime graph components of $P$, we
proceed case by case analysis.

\begin{ca} \  $s(P)=2$.
\end{ca}
In this case we have $m_2(P)=2^p-1$.
\begin{itemize}
\item[{$(1)$}] {\em The simple group $P$ is isomorphic to none of the
simple groups $C_n(q)$, $n=2^m\geqslant 2$; $D_r(q)$, $r\geqslant
5$, $q=2, 3, 5$; $D_{r+1}(q)$, $q=2, 3$;  $F_4(q)$, $q$ odd;
$G_2(q)$, $q\equiv \pm 1\pmod{3}$; $^2D_r(3)$, $r\geqslant 5$,
$r\neq 2^n+1$; $^2D_n(2)$, $n=2^m+1\geqslant 5$; $^2D_n(3)$,
$9\leqslant n=2^m+1 \neq r$; $^3D_4(q)$, $C_r(3)$, $B_r(3)$;
$B_n(q)$, $n=2^m\geqslant 4$, $q$ odd, $^2A_3(2)$ and
$^2F_4(2)'$.}

\subitem {(1.1)} \ If $P\cong C_n(q)$, $n=2^m\geqslant 2$, then
\[ |P|=|C_n(q)|=m_1\times m_2=q^{n^2}(q^n-1)\prod_{i=1}^{n-1}(q^{2i}-1)\times \frac{q^n+1}{2}.
\]
Because $\frac{q^n+1}{2}=2^p-1$, it implies that
$q^n-1=4(2^{p-1}-1)$. Evidently $p\neq 7$. On the other hand,
since $(q^n-1)^2$ divides $|P|$, thus $(2^{p-1}-1)^2$ must divides
$|G|$. This is a contradiction by Lemma \ref{m}.

\subitem {(1.2)} \  If $P\cong D_r(q)$, $r\geqslant 5$, $q=2, 3,
5$, then
\[ |P|=|D_r(q)|=m_1\times m_2=q^{r(r-1)}\prod_{i=1}^{r-1}(q^{2i}-1)\times \frac{q^r-1}{(q-1,4)} \ .
\]
In this case we have $\frac{q^r-1}{(q-1,4)}=2^p-1$. If $q=2$,
then $2^r-1=2^p-1$, and hence $r=p$. Thus $2^{p(p-1)}$ divides
$|P|$ and so $|G|$, which is a contradiction. If $q=3$, then we
obtain $2^2(2^{p-1}-1)=3(3^{r-1}-1)$ and if $q=5$ then we get
$2^3(2^{p-1}-1)=5(5^{r-1}-1)$. In both cases, we easy to see that
$(2^{p-1}-1)^2\big ||G|$, which contradicts Lemma \ref{m}.

\subitem {(1.3)} \ If $P\cong F_4(q)$, $q$ odd, then we have
\[|P|=|F_4(q)|=m_1\times m_2=q^{24}(q^4-1)(q^6-1)^2(q^8-1)\times (q^{4}-q^2+1).\]
Now, from $q^4-q^2+1=2^p-1$ we deduce that
$2(2^{p-1}-1)=q^2(q^2-1)$. But then, we have $(2^{p-1}-1)^2\big |
|G|$, which is again a contradiction by Lemma \ref{m}.

The other cases are settled similarly.

\item[{$(2)$}] {\em The simple group $P$ is isomorphic to none of the
simple groups $^2D_n(q)$, $n=2^m\geqslant 4$, and $C_r(2)$.}

\subitem {(2.1)} \ If $P\cong {^2D_n(q)}, n=2^m\geqslant 4$, then
\[ |P|=|^2D_n(q)|=m_1\times m_2=q^{n(n-1)}\prod_{i=1}^{n-1}(q^{2i}-1)\times \frac{q^n+1}{(2,q+1)}. \]
Moreover, we have $\frac{q^n+1}{(2, q+1)}=2^p-1$. We now consider
two cases separately.

$(a)$ $(2, q+1)=1$. In this case, we get $q^n=2(2^{p-1}-1)$, an
impossible.

$(b)$ $(2, q+1)=2$. In this case, we obtain $q^n-1=4(2^{p-1}-1)$,
and since $(q^n-1)^2$ divides $|P|$, it follows that
$(2^{p-1}-1)^2$ divides $|G|$, which is impossible by Lemma
\ref{m}.

\subitem {(2.2)} \ If $P\cong C_r(2)$, then
\[ |P|=|C_r(2)|=m_1\times m_2=2^{r^2}(2^r+1)\prod_{i=1}^{r-1}(2^{2i}-1)\times (2^r-1). \]
From $2^r-1=2^p-1$, it follows $r=p$. But then, we must have
$2^{p^2}\big ||G|$, which is a contradiction.

\item[{$(3)$}] {\em The simple group $P$ is isomorphic to none of the
simple groups $A_{r-1}(q)\cong L_r(q)$, $(r,q)\neq (3,2), (3,4)$;
$A_r(q)\cong L_{r+1}(q)$, $q-1|r+1$; $^2A_{r-1}(q)$ and
$^2A_r(q)$, $q+1|r+1$, $(r, q)\neq (3,3), (5,2)$, where $r$ is an
odd prime.}

Since the proofs of all cases are similar, only the proofs for
the simple groups $A_{r-1}(q)$, $(r,q)\neq (3,2), (3,4)$ and
$A_r(q)$ with $q-1|r+1$, are presented.

\subitem {(3.1)} \ If $P\cong A_{r-1}(q)\cong L_{r}(q)$,  $(r,
q)\neq (3,2), (3,4)$, then
$$|P|=m_1\times m_2=q^{r\choose 2}\prod_{i=1}^{r-1}(q^i-1)\times \frac{q^r-1}{(r, q-1)(q-1)},$$
and
$$\frac{q^r-1}{(r, q-1)(q-1)}=2^p-1.$$
Let $q=s^f$. If $s=2$ and $f>1$, then
$$2^{fr}-1\gneqq \frac{2^{fr}-1}{(r, 2^f-1)(2^f-1)}=2^p-1.$$
Since $2^{fr}-1$ divides $|P|$, and so $|G|$, we get a
contradiction by Theorem \ref{zsig}. In the case $s=2$ and $f=1$,
we obtain that $r=p$, and so $P\cong A_{p-1}(2)\cong L_p(2)$, as
required.

In the sequel we assume that $s$ is an odd prime. First of all,
we have
$$q^r-1\gneqq \frac{q^r-1}{(r, q-1)(q-1)}=2^p-1,$$
or equivalently $q^r>2^p$. Since $s\in \pi(G)$, we assume that
$s\in {\rm ppd}(2^k-1)$, for some $k$.  Let $(2^k-1)_s=s^m$,
where $m$ is a natural number and
$$a:=(2^k-1)(2^{2k}-1)\cdots (2^{[\frac{p-1}{k}]k}-1),$$
(Note that $k$ divides $p-1$, and so
$[\frac{p-1}{k}]=\frac{p-1}{k}$). Then, by Lemma
\ref{exp-element}, we have
$$\begin{array}{lll}a_s&=&\prod\limits_{l=1}^{\frac{p-1}{k}}(2^{kl}-1)_s=
\prod\limits_{l=1}^{\frac{p-1}{k}}l_s(2^k-1)_s=
\prod\limits_{l=1}^{\frac{p-1}{k}}l_ss^m=s^{\frac{p-1}{k}m}
\prod\limits_{l=1}^{\frac{p-1}{k}}l_s\\[0.3cm]
& =
&s^{\frac{p-1}{k}m}\left(\prod\limits_{l=1}^{\frac{p-1}{k}}l\right)_s=
s^{\frac{p-1}{k}m}((\frac{p-1}{k})!)_s= s^{\frac{p-1}{k}m}\cdot
s^{\sum_{j=1}^{\infty}[\frac{p-1}{ks^j}]}=s^{\frac{p-1}{k}m+
\sum_{j=1}^{\infty}[\frac{p-1}{ks^j}]},\end{array}$$ and since
$|G|_s=a_s$, it follows that
\begin{equation}\label{e2} s^{f\frac{r(r-1)}{2}}=|P|_s\leqslant |G|_s=
s^{\frac{p-1}{k}m+\sum_{j=1}^{\infty}[\frac{p-1}{ks^j}]}.
\end{equation}
On the other hand, we have
$$\sum_{j=1}^{\infty}\left[\frac{p-1}{ks^j}\right]\leqslant
\sum_{j=1}^{\infty}\frac{p-1}{ks^j}=\frac{p-1}{k}
\sum_{j=1}^{\infty}\frac{1}{s^j}=\frac{p-1}{k}\cdot
\frac{1}{s-1}\leqslant  \frac{p-1}{k}.
$$
If this is substituted in $(\ref{e2})$ and noting that $q^r>2^p$,
then we obtain
$$\begin{array}{lll} s^{f\frac{r(r-1)}{2}} & \leqslant &
s^{\frac{p-1}{k}(m+1)}=(s^m)^{\frac{p-1}{k}}\cdot s^{\frac{p-1}{k}}\\[0.3cm]
& < & (2^k-1)^{\frac{p-1}{k}}\cdot
(2^k-1)^{\frac{p-1}{k}}=(2^k-1)^{2\frac{p-1}{k}}
\\[0.3cm]
& < & (2^k)^{2\frac{p-1}{k}}=2^{2(p-1)}<(2^p)^2<(q^r)^2=s^{2fr},
\end{array}$$
which implies that $\frac{r(r-1)}{2}<2r$, and so $r=3$. Thus
$P\cong L_3(q)$, $q=s^f\neq 2, 4$, and
\begin{equation}\label{eee3}
\frac{q^3-1}{(3, q-1)(q-1)}=2^p-1.\end{equation} Note that $|{\rm
Out}(P)|=(3, q-1)\cdot f\cdot 2$, and hence
$$|{\rm Aut}(P)|=|P|\cdot |{\rm Out}(P)|=q^2(q^2-1)(q^3-1)\cdot
f\cdot 2.$$ Moreover, subtracting 1 from both sides of Eq.
$(\ref{eee3})$ and easy computations show that
$$(q-1)(q+2)=\left\{\begin{array}{lll} 4(2^{p-2}-1) & {\rm if} & (3, q-1)=1,\\[0.3cm]
6(2^{p-1}-1) & {\rm if} & (3, q-1)=3. \end{array}\right.$$ In
what follows, we will consider two cases separately.

\subsubitem {{\em Case} 1.} $(3, q-1)=1$. Let $t\in {\rm
ppd}(2^{p-2}-1)$ and $(2^{p-2}-1)_t=t^m$. Since $(q-1, q+2)=1$,
we conclude that $(q-1)_t=t^m$ or $(q+2)_t=t^m$. If
$(q-1)_t=t^m$, then since $(q-1)^2$ divides the order of $P$, it
follows that $t^{2m}\big| |P|$, and so $t^{2m}\big||G|$. But this
contradicts Lemma \ref{m}, because $(2^{p-2}-1)^2$ does not
divide the order of $G$. Therefore we may assume that
$(q+2)_t=t^m$. Clearly $t\notin \pi(P)$, since $$(q+2, \
q^2)=(q+2, \ q^2-1)=(q+2, \ q^3-1)=1.$$ On the other hand, since
$t\in {\rm ppd}(2^{p-2}-1)$ and $t|2^{t-1}-1$, we deduce that
$p-2|t-1$, and so $t\geqslant p-1$. In addition, from
$(q^3-1)/(q-1)=2^p-1$, it follows that $$q(q+1)=2(2^{p-1}-1),$$
and so $q=s^f|2^{p-1}-1$, since $(q, 2)=1$. Thus $f\leqslant
p-2<p-1\leqslant t$, which implies that $t\nmid f$. By what
observed above we see that $t\notin \pi({\rm Aut}(P))$, and so
$t\in \pi(K)$. Assume now that $R\in {\rm Syl}_t(K)$. Certainly
$R\in {\rm Syl}_t(G)$, and since $K$ is nilpotent,
$R\trianglelefteq G$. Now a $(2^p-1)$-Sylow subgroup of $G$, say
$T$, acts fixed point freely on $R$ by conjugation. This shows
that the group $RT$ is a Frobenius group with kernel $R$ and
complement $T$, and so
$$2^p-1\leqslant |R|-1\leqslant 2^{p-2}-1,$$ which is a
contradiction.

\subsubitem {{\em Case} 2.} $(3, q-1)=3$. The proof goes in the
same way as previous case.

\subitem {(3.2)} \  If $P\cong A_r(q)\cong L_{r+1}(q)$, $q-1|r+1$,
then
$$|P|=m_1\times m_2=q^{r+1\choose 2}(q^{r+1}-1)\prod_{i=2}^{r-1}(q^i-1)\times \frac{q^r-1}{q-1},$$
and
$$\frac{q^r-1}{q-1}=2^p-1.$$
Subtracting 1 from both sides of this equality, we obtain
$$q(q^{r-1}-1)=2(q-1)(2^{p-1}-1).$$
If $q$ is even, then $q=2$ and $r=p$, which implies that $P\cong
L_{p+1}(2)$, as required. Therefore, we may assume that $q=s^f$,
where $s$ is an odd prime and $f\geqslant 1$ a natural number.
First of all, we have
$$q^r-1>\frac{q^r-1}{q-1}=2^p-1,$$
or equivalently $q^r>2^p$. Since $s\in \pi(G)$, we assume that
$s\in {\rm ppd}(2^k-1)$, for some $k$.  Let $|2^k-1|_s=s^m$,
where $m$ is a natural number and
$$a:=(2^k-1)(2^{2k}-1)\cdots (2^{[\frac{p-1}{k}]k}-1),$$
(Note that $k$ divides $p-1$, and so
$[\frac{p-1}{k}]=\frac{p-1}{k}$). Then, by Lemma
\ref{exp-element}, we have
$$\begin{array}{lll}a_s&=&\prod\limits_{l=1}^{\frac{p-1}{k}}(2^{kl}-1)_s=
\prod\limits_{l=1}^{\frac{p-1}{k}}l_s(2^k-1)_s=
\prod\limits_{l=1}^{\frac{p-1}{k}}l_ss^m=s^{\frac{p-1}{k}m}
\prod\limits_{l=1}^{\frac{p-1}{k}}l_s\\[0.3cm]
& =
&s^{\frac{p-1}{k}m}\left(\prod\limits_{l=1}^{\frac{p-1}{k}}l\right)_s=
s^{\frac{p-1}{k}m}((\frac{p-1}{k})!)_s= s^{\frac{p-1}{k}m}\cdot
s^{\sum_{j=1}^{\infty}[\frac{p-1}{ks^j}]}=s^{\frac{p-1}{k}m+
\sum_{j=1}^{\infty}[\frac{p-1}{ks^j}]},\end{array}$$ and since
$|G|_s=a_s$, it follows that
\begin{equation}\label{e4} s^{f\frac{r(r+1)}{2}}=|P|_s\leqslant |G|_s=
s^{\frac{p-1}{k}m+\sum_{j=1}^{\infty}[\frac{p-1}{ks^j}]}.
\end{equation}
On the other hand, we have
$$\sum_{j=1}^{\infty}\left[\frac{p-1}{ks^j}\right]\leqslant
\sum_{j=1}^{\infty}\frac{p-1}{ks^j}=\frac{p-1}{k}
\sum_{j=1}^{\infty}\frac{1}{s^j}=\frac{p-1}{k}\cdot
\frac{1}{s-1}\leqslant  \frac{p-1}{k}.
$$
If this is substituted in $(\ref{e4})$ and noting that $q^r>2^p$,
then we obtain
$$\begin{array}{lll} s^{f\frac{r(r+1)}{2}} & \leqslant &
s^{\frac{p-1}{k}(m+1)}=(s^m)^{\frac{p-1}{k}}\cdot s^{\frac{p-1}{k}}\\[0.3cm]
& < & (2^k-1)^{\frac{p-1}{k}}\cdot
(2^k-1)^{\frac{p-1}{k}}=(2^k-1)^{2\frac{p-1}{k}}
\\[0.3cm]
& < & (2^k)^{2\frac{p-1}{k}}=2^{2(p-1)}<(2^p)^2<(q^r)^2=s^{2fr},
\end{array}$$
which implies that $\frac{r(r+1)}{2}<2r$, a contradiction.

\item[{$(4)$}] {\em The simple group $P$ is isomorphic to none of the
simple groups $E_6(q)$ and $^2E_{6}(q)$, $q>2$.}

\subsubitem {(4.1)} \ If $P$ is isomorphic to $E_6(q)$, $q=s^f$,
then
$$ |P|=|E_6(q)|=m_1\times m_2=q^{36}(q^{12}-1)(q^{8}-1)(q^{6}-1)(q^{5}-1)(q^{3}-1)(q^2-1)\times \frac{q^6+q^3+1}{(3, q-1)},$$
and $$\frac{q^6+q^3+1}{(3, q-1)}=2^p-1.$$ Thus, we have
$$q^9-1>\frac{q^9-1}{q^3-1}=q^6+q^3+1=(3, q-1)\cdot (2^p-1)\geqslant 2^p-1,$$
which yields that $q^9> 2^p$. Again since $s\in \pi(G)$, we assume
that $s\in {\rm ppd}(2^k-1)$, for some $k$. Suppose
$|2^k-1|_s=s^m$, where $m$ is a natural number and
$$a:=(2^k-1)(2^{2k}-1)\cdots (2^{[\frac{p-1}{k}]k}-1).$$
Similarly to the previous case, we obtain
$$|G|_s=a_s=s^{\frac{p-1}{k}m+\sum_{j=1}^{\infty}[\frac{p-1}{ks^j}]},$$
and hence $$\begin{array}{lll} s^{36\cdot f}=|P|_s &\leqslant &
|G|_s=s^{\frac{p-1}{k}m+\sum_{j=1}^{\infty}[\frac{p-1}{ks^j}]}\\[0.3cm]
&\leqslant & s^{\frac{p-1}{k}(m+1)}=(s^m)^{\frac{p-1}{k}}\cdot
s^{\frac{p-1}{k}}\\[0.3cm]
& < & (2^k-1)^{\frac{p-1}{k}}\cdot
(2^k-1)^{\frac{p-1}{k}}=(2^k-1)^{2\frac{p-1}{k}}
\\[0.3cm]
& < & (2^k)^{2\frac{p-1}{k}}=2^{2(p-1)}<(2^p)^2<s^{18\cdot f},
\end{array}$$
which is a contradiction. \subsubitem {(4.2)} The case when
$P\cong {^2E}_{6}(q)$, $q>2$, is similar to the previous case.
\begin{ca}
$s(P)=3.$
\end{ca}
In this case we have $2^p-1\in \{m_2(P), m_3(P)\}$.
\subitem{$(1)$} $P\cong L_2(q)$, $4|q+1$. In this case
$\frac{q-1}{2}=2^p-1$ or $q=2^p-1$. The first case is obviously
impossible, since we obtain $q=2^{p+1}-1$ which must divides
$|G|$. For the latter case, we first notice that $q$ is a
Mersenne prime and $$|P|=|L_2(q)|=\frac{1}{(2,
q-1)}q(q^2-1)=2^p(2^{p-1}-1)(2^p-1).$$ Moreover, since
$P\leqslant G/K \leqslant {\rm Aut}(P)$ and $|{\rm Aut}(P):P|=2$
we deduce that $2^{p-2}-1$ divides $|K|$. Let $r\in {\rm
ppd}(2^{p-2}-1)$. Now we consider the Sylow $r$-subgroup $R$ of
$K$. Evidently $R\in {\rm Syl}_r(G)$ and $R\lhd G$ because $K$ is
a nilpotent subgroup. Now if $Q\in {\rm Syl}_q(G)$, then $Q$ acts
on $R$ by conjugation and this action is fixed point free. Hence
$RQ$ is a Frobenius group with kernel $R$ and complement $Q$, and
we must have $$q=2^p-1\leqslant|R|-1\leqslant 2^{p-2}-1,$$ which
is a contradiction.

\subitem{$(2)$} $P\cong L_2(q)$, $4|q-1$. In this case we must
have $q=2^p-1$ or $\frac{q+1}{2}=2^p-1$. The first case is
obviously impossible, because $q-1=2(2^{p-1}-1)$ and so $4\nmid
q-1$. If $\frac{q+1}{2}=2^p-1$, then $q=2^{p+1}-3$, and so
$$|P|=|L_2(q)|=\frac{1}{(2, q-1)}q(q^2-1)=2^2(2^{p-1}-1)(2^p-1)(2^{p+1}-3).$$
Let $q=2^{p+1}-3=s^f$, where $s$ is a prime number. Evidently
$s\geqslant 5$, and so
$$2^{p+1}\geqslant 2^{p+1}-3=s^f\geqslant 5^f\geqslant 2^{2f},$$
which forces $f\leqslant \frac{p+1}{2}$. Moreover, since $$|{\rm
Out}(P)|=|{\rm Aut}(P):P|=(2, q-1)\cdot f=2\cdot f,$$ it follows
that  $$|{\rm Aut}(P)|=2^3(2^{p-1}-1)(2^p-1)(2^{p+1}-3)\cdot f.$$
Let $r\in {\rm ppd}(2^{p-2}-1)\subset \pi(G)$. Now, we claim that
$(r, |{\rm Aut}(P)|)=1$. Indeed, on the one hand, we have
$$\big(2^{p-2}-1, \ \  2^3(2^{p-1}-1)(2^{p}-1)(2^{p+1}-3)\big)=1,$$
whose validity is verified by direct computations. On the other
hand, since $r|2^{r-1}-1$, we deduce that $p-2|r-1$, and so
$r\geqslant p-1$. Combining this with the inequality $f\leqslant
\frac{p+1}{2}$, we obtain
$$f\leqslant \frac{p+1}{2}<p-1\leqslant r,$$
which yields that $(r, f)=1$. This completes the proof of our
claim.

Therefore, from $(r, |{\rm Aut}(P)|)=1$, it follows that $r\in
\pi(K)$. As previous case, we consider the Sylow $r$-subgroup $R$
of $K$, which is also the normal Sylow $r$-subgroup of $G$. Now a
$(2^p-1)$-Sylow subgroup of $G$, say $Q$, acts fixed point freely
on $R$ by conjugation. This shows that the group $RQ$ is a
Frobenius group with kernel $R$ and complement $Q$, and so
$$2^{p}-1\leqslant|R|-1\leqslant 2^{p-2}-1,$$ which is a
contradiction.

\subitem{$(3)$} $P\cong L_2(q)$, $4|q$. Here, we must have
$q-1=2^p-1$ or $q+1=2^p-1$. In the first case, we obtain
$q+1=2^p+1\big | |G|$, an impossible by Theorem \ref{zsig}. In
the second case, we get $q=2(2^{p-1}-1)$, which is again a
contradiction.

\subitem{$(4)$} $P\cong G_2(q)$, $3|q$. In this case
$q^2-q+1=2^p-1$ or $q^2+q+1=2^p-1$. Now by easy calculate in both
cases we obtain that $(2^{p-1}-1)^2\big | |G|$, which is a
contradiction by Lemma \ref{m}.

\subitem{$(5)$} $P\cong {^2G_2(q)}$, $q=3^{2n+1}$. In this case,
we have
$$3^{2n+1}-3^{n+1}+1=2^p-1 \ \ \ \mbox{or} \ \ \
3^{2n+1}+3^{n+1}+1=2^p-1.$$ Assume $3^{2n+1}-3^{n+1}+1=2^p-1$.
Now we easily deduce that
$$2(2^{\frac{p-1}{2}}-1)(2^{\frac{p-1}{2}}+1)=3^{n+1}(3^n-1).$$ If
$3^{n+1}$ divides $2^{\frac{p-1}{2}}-1$, then
$$3^n-1<3^{n+1}\leqslant 2^{\frac{p-1}{2}}-1<
2^{\frac{p-1}{2}}+1.$$ Hence, we obtain
$$3^{n+1}(3^n-1)<2(2^{\frac{p-1}{2}}-1)(2^{\frac{p-1}{2}}+1),$$
which is a contradiction. Assume now that $3^{n+1}$ divides
$2^{\frac{p-1}{2}}+1$. Then $2^{\frac{p-1}{2}}+1=k(3^{n+1})$, for
some $k$, and so $3^{n+1}\leqslant 2^{\frac{p-1}{2}}+1$. On the
other hand, we observe that $2k(2^{\frac{p-1}{2}}-1)=3^n-1$, and
hence $3^n-1\geqslant 2(2^{\frac{p-1}{2}}-1)$, i.e., $3^n\geqslant
2^{\frac{p+1}{2}}-1$. Therefore, we have
$$2^{\frac{p+1}{2}}-1\leqslant 3^n<3^{n+1}\leqslant
2^{\frac{p-1}{2}}+1,$$ which is a contradiction. For other case
the discussion is similar.

\subitem{$(6)$} $P\cong {^2D_r(3)}$, $r=2^n+1\geqslant 3$. For
this case, we have $\frac{3^r+1}{4}=2^p-1$ or
$\frac{3^{r-1}+1}{2}=2^p-1$. In the first case, we obtain
$2^2(2^p+1)=3^2(3^{r-2}+1)$. Now, we consider a primitive prime
divisor $r\in {\rm ppd}(2^{2p}-1)$. Then $r\in \pi(2^p+1)\subset
\pi(P)$ and $r\notin \pi(G)$, which is a contradiction. In the
second case, we get $2^{p+1}=3(3^{r-2}+1)$, which is a
contradiction.

\subitem{$(7)$} $P\cong F_4(q)$, $2|q$. In this case we must have
$$q^4+1=2^p-1 \ \ \ \mbox{or} \ \ \  q^4-q^2+1=2^p-1.$$ The first
case obviously is impossible. In the latter case, we deduce
$$q^2(q^2-1)=2(2^{p-1}-1),$$ and so $(2^{p-1}-1)^2$ divides $|G|$,
which is a contradiction.

\subitem{$(8)$} $P\cong {^2F_4(q)}$, $q=2^{2m+1}>2$. Then
$$2^{2(2m+1)}-2^{3m+2}+2^{2m+1}-2^{m+1}+1=2^p-1,$$ or
$$2^{2(2m+1)}+2^{3m+2}+2^{2m+1}+2^{m+1}+1=2^p-1.$$ Now, it is not
difficult to see that any of equalities cannot hold.

\subitem{$(9)$} If $P\cong {^2A_5(2)}$ or $E_7(3)$, then
$2^p-1=7$, 11, 757 or 1093, which is a contradiction. If $P\cong
E_7(2)$ then $2^p-1=127$ and $p=7$. But then we must have $43\big
| |G|$ which is a contradiction.
\end{itemize}
\begin{ca}
$s(P)=4, 5.$
\end{ca}

In this case we have $2^p-1\in \{m_2(P), m_3(P), m_4(P),
m_5(P)\}$.

\begin{itemize}
\item[{$(1)$}] The cases $P \cong  A_2(4)$, ${^2E_6(2)}$ are clearly
impossible.

\item[{$(2)$}] If $P \cong {^2B_2(2^{2m+1}})$, $m\geqslant 1$, and $2^{2m+1} \pm
2^{m+1} +1= 2^p -1$, then $m=0$, against the fact $m \ge 1$. In
the case when $2^{2m+1}-1=2^p-1$, it follows that $p=2m+1$ and we
obtain $2^{2p}+1\big | |G|$, which is a contradiction.

\item[{$(3)$}] If $P\cong E_8(q)$, then $2^p -1$ is one of the following:

\subitem{$(i)$} $q^8-q^7+q^5-q^4+q^3-q+1$. This implies that
$$2(2^{p-1}-1)=q(q-1)(q+1)(q^5-q^4+q^3+1),$$ which contradicts the
fact that $8$ divides $(q^2-1)$, if $q$ is odd. If $q$ is even,
then $q=2$ also gives a contradiction.

\subitem{$(ii)$} $q^8+q^7-q^5-q^4-q^3+q+1$. This implies that
$$2(2^{p-1} -1) = q(q-1)(q+1)(q^5 +q^4 +q^3 -1),$$ which
contradicts the fact that $8$ divides $(q^2-1)$, if $q$ is odd.
If $q$ is even, then $q=2$ also gives a contradiction.

\subitem{$(iii)$} $q^8-q^6+q^4-q^2+1$. This implies that
$$2(2^{p-1}-1)=q^2(q-1)(q+1)(q^4+q^2-1),$$ which contradicts the
fact that $8$ divides $(q^2-1)$, if $q$ is odd. If $q$ is even,
then $q^2=2$ also gives a contradiction.

\subitem{$(iv)$} $q^8-q^4 +1$. This implies that
$2(2^{p-1}-1)=q^4(q^4-1)$, which also gives a contradiction.
\end{itemize}
The proof of this lemma is complete. \end{proof}

\begin{lm}\label{automL} $G$ is isomorphic to ${\rm Aut}(L)$.
\end{lm}
\begin{proof} By Lemma \ref{isomorL}, $P$ is isomorphic to $L$, and so
$L\leqslant G/K\leqslant {\rm Aut}(L)$. Since $|{\rm
Out}(L))|=2$, $G/K\cong L$ or $G/K\cong {\rm Aut}(L)$. In the
first case, $|K|=2$ and so $K\leqslant Z(G)$ which forces $G$
possesses an element of order $2\cdot (2^p-1)$, a contradiction.
In the later case, one can easily deduce that $K=1$ and $G\cong
{\rm Aut}(L)$, as required.
\end{proof}

The proof of the theorem is complete.
\end{proof}
\section{Appendix}
In a series of papers, it was shown that many finite simple groups
are OD-characterizable or 2-fold OD-characterizable. Table 4 lists
finite simple groups which are currently known to be $k$-fold
OD-characterizable for $k\in \{1, 2\}$.

\begin{center}
{\bf Table 4}. Some non-abelian simple groups $S$ with $h_{\rm
OD}(S)=1$ or $2$.\\[0.4cm]
$\begin{array}{l|l|c|l} \hline S & {\rm Conditions \ on} \ S&
h_{\rm OD} & {\rm Refs.} \\ \hline
 \mathbb{A}_n & \ n=p, p+1, p+2 \ (p \ {\rm a \ prime})& 1 &  \cite{mz}, \cite{mzd}    \\
 & \ 5\leqslant n\leqslant 100, n\neq 10   & 1 & \cite{hm}, \cite{kogani}, \cite{banoo-alireza},  \\
 & & & \cite{mz3}, \cite{zs-new2} \\
& \ n=106, \ 112 & 1 &      \cite{yan-chen}       \\
& \ n=10 & 2 &      \cite{mz2}       \\[0.2cm]
L_2(q) &  q\neq 2, 3& 1 &    \cite{mz}, \cite{mzd},\\
& & &  \cite{zshi} \\[0.1cm]
L_3(q) &  \ |\pi(\frac{q^2+q+1}{d})|=1, \ d=(3, q-1) & 1 &   \cite{mzd} \\[0.2cm]
U_3(q) &  \ |\pi(\frac{q^2-q+1}{d})|=1, \ d=(3, q+1), q>5 & 1 &   \cite{mzd} \\[0.2cm]
L_4(q) &  \ q\leqslant 17  & 1 &   \cite{bakbari, amr} \\[0.1cm]
L_3(9) & & 1 & \cite{zs-new4}\\[0.1cm]
U_3(5) &   & 1 &   \cite{zs-new5} \\[0.1cm]
U_4(7) &   & 1 &   \cite{amr} \\[0.1cm]
L_n(2) & \ n=p \ {\rm or} \ p+1, \ {\rm for \ which} \ 2^p-1 \ {\rm is \ a \ prime} & 1 & \cite{amr} \\[0.2cm]
L_9(2) &   & 1 &   \cite{khoshravi} \\[0.1cm]
R(q) & \ |\pi(q\pm \sqrt{3q}+1)|=1, \ q=3^{2m+1}, \ m\geqslant 1 & 1 & \cite{mzd} \\[0.2cm]
{\rm Sz} (q) & \ q=2^{2n+1}\geqslant 8& 1 &   \cite{mz}, \cite{mzd} \\[0.2cm]
B_m(q), C_m(q) &  m=2^f\geqslant 4,  \
|\pi\big((q^m+1)/2\big)|=1, \
 & 2 & \cite{akbari-moghadam}\\[0.2cm]
B_2(q)\cong C_2(q) &  \ |\pi\big((q^2+1)/2\big)|=1, \ q\neq
3 & 1 & \cite{akbari-moghadam}\\[0.2cm]
B_m(q)\cong C_m(q) &  m=2^f\geqslant 2, \ 2|q, \ |\pi\big(q^m+1\big)|=1, \ (m, q)\neq (2, 2) & 1 &
\cite{akbari-moghadam}\\[0.2cm]
B_p(3), C_p(3) &  |\pi\big((3^p-1)/2\big)|=1, \  p \ {\rm is \ an
\ odd \
prime}  & 2 & \cite{akbari-moghadam}, \cite{mzd}\\[0.2cm]
B_3(5), C_3(5) & & 2 & \cite{akbari-moghadam} \\[0.2cm]
C_3(4) & & 1 & \cite{moghadam} \\[0.1cm]
S &  \ \mbox{A sporadic simple group} & 1 & \cite{mzd} \\[0.1cm]
S &  \ \mbox{A simple group with} \ |\pi(S)|=4, \ \ S\neq \mathbb{A}_{10} & 1 & \cite{zs} \\[0.1cm]
S &  \  \mbox{A simple group with} \ |S|\leqslant 10^8, \ \ S\neq \mathbb{A}_{10}, \ U_4(2) & 1 & \cite{ls} \\[0.1cm]
S &  \  \mbox{A simple $C_{2,2}$- group} & 1 & \cite{mz}
\end{array}$
\end{center}
\footnotetext{In Table 4, $q$ is a power of a prime number.}

Although we have not found a simple group which is $k$-fold
OD-characterizable for $k\geqslant 3$, but among non-simple
groups, there are many groups which are $k$-fold
OD-characterizable for $k\geqslant 3$. As an easy example, if $P$
is a $p$-group of order $p^n$, then $h_{\rm OD}(P)=\nu(p^n)$. In
connection with such groups, Table 5 lists finite non-solvable
groups which are currently known to be OD-characterizable or
$k$-fold OD-characterizable with $k\geqslant 2$.

\begin{center}
{\bf Table 5}. Some non-solvable groups $G$ with certain $h_{\rm
OD}(G)$.\\[0.4cm]
$\begin{array}{l|l|c|l} \hline G & {\rm Conditions \ on} \ G &
h_{\rm OD}(G) & {\rm Refs.} \\ \hline
{\rm Aut}(M) & M \ \mbox{is a sporadic group}  \neq  J_2, M^cL    & 1 & \cite{mz} \\[0.1cm]
\mathbb{S}_n & n=p, \ p+1 \ (p\geqslant 5 \ \mbox{is a prime})& 1 &  \cite{mz}    \\[0.1cm]
M &  M\in \mathcal{C}_1 & 2 &      \cite{mz2}       \\[0.1cm]
M & M\in \mathcal{C}_2 & 2 &  \cite{mzd} \\[0.1cm]
M & M\in \mathcal{C}_3 & 8 &      \cite{mz2}       \\[0.1cm]
M & M\in \mathcal{C}_4  & 3 & \cite{hm, kogani, banoo-alireza, mz3, yan-chen} \\[0.1cm]
M & M\in \mathcal{C}_5     & 2 & \cite{mz2} \\[0.1cm]
M & M\in \mathcal{C}_6    & 3 & \cite{mz2} \\[0.1cm]
M & M\in \mathcal{C}_7  & 6 &\cite{banoo-alireza} \\[0.1cm]
M & M\in \mathcal{C}_8 & 1 &  \cite{zs-new1} \\[0.1cm]
M & M\in \mathcal{C}_9  & 9 &\cite{zs-new1} \\[0.1cm]
M & M\in \mathcal{C}_{10}  & 1 &\cite{zs-new5} \\[0.1cm]
M & M\in \mathcal{C}_{11} & 3 &  \cite{zs-new5} \\[0.1cm]
M & M\in \mathcal{C}_{12}  & 6 &\cite{zs-new5}\\[0.1cm]
M & M\in \mathcal{C}_{13}  & 1 &\cite{y-chen-w}
\end{array}$
\end{center}
\begin{tabular}{lll}
$\mathcal{C}_1$ & $\!\!\!\!=$  & $ \!\!\!\! \{\mathbb{A}_{10}, J_2\times \mathbb{Z}_{3} \}$\\[0.1cm]
$\mathcal{C}_2$ & $\!\!\!\!=$ & $\!\!\!\! \{S_6(3) , O_7(3) \}$\\[0.1cm]
$\mathcal{C}_3$ & $\!\!\!\!=$ & $\!\!\!\! \{ \mathbb{S}_{10},  \
\mathbb{Z}_{2}\times \mathbb{A}_{10}, \ \mathbb{Z}_2\cdot
\mathbb{A}_{10}, \ \mathbb{Z}_6\times J_2,  \ \mathbb{S}_3 \times
J_2, \ \mathbb{Z}_3\times (\mathbb{Z}_2\cdot J_2),
$ \\[0.1cm]
& & $(\mathbb{Z}_3\times J_2)\cdot \mathbb{Z}_2, \ \mathbb{Z}_3\times {\rm Aut}(J_2)\}.$\\[0.1cm]
$ \mathcal{C}_4$ & $\!\!\!\!=$ & $\!\!\!\! \{\mathbb{S}_{n}, \
\mathbb{Z}_{2}\cdot \mathbb{A}_{n}, \ \mathbb{Z}_{2}\times
\mathbb{A}_{n}\}$, \
where $9\leqslant n\leqslant 100$ with $n\neq 10, p, p+1$ ($p$ a prime)\\[0.1cm]
& & or $n=106, \ 112$.\\[0.1cm]
$\mathcal{C}_5$ & $\!\!\!\!=$ & $\!\!\!\! \{ {\rm Aut}(M^cL), \ \mathbb{Z}_2\times M^cL\}$.\\[0.1cm]
$\mathcal{C}_6$ & $\!\!\!\!=$ & $\!\!\!\! \{{\rm Aut}(J_2), \
\mathbb{Z}_2\times J_2, \
\mathbb{Z}_2\cdot J_2\}.$\\[0.1cm]
$\mathcal{C}_7$ & $\!\!\!\!=$ & $\!\!\!\! \{{\rm Aut}(S_6(3)), \
\mathbb{Z}_2\times S_6(3), \  \mathbb{Z}_2\cdot S_6(3),  \
\mathbb{Z}_2\times O_7(3), \
\mathbb{Z}_2\cdot O_7(3)$, \ ${\rm Aut}(O_7(3))\}$.\\[0.1cm]
$\mathcal{C}_8$ & $\!\!\!\!=$ & $\!\!\!\! \{L_2(49):2_1, \
L_2(49):2_2, \
L_2(49):2_3\}$.  \\[0.1cm]
$\mathcal{C}_9$ & $\!\!\!\!=$ & $\!\!\!\! \{L\cdot 2^2,  \
\mathbb{Z}_{2}\times (L:2_1), \ \mathbb{Z}_2\times (L:2_2), \
\mathbb{Z}_2\times (L\cdot 2_3),  \ \mathbb{Z}_2\cdot(L:2_1), $
\\[0.1cm]
& & $\mathbb{Z}_2\cdot(L:2_2), \ \mathbb{Z}_2\cdot(L\cdot2_3), \
\mathbb{Z}_4\times L, \ (\mathbb{Z}_2\times \mathbb{Z}_2)\times
L\}$, \ where $L=L_2(49)$.\\[0.1cm]
$\mathcal{C}_{10}$ & $\!\!\!\!=$  & $ \!\!\!\! \{U_3(5), \ U_3(5):2\}$\\[0.1cm]
$\mathcal{C}_{11}$ & $\!\!\!\!=$  & $ \!\!\!\! \{U_3(5):3, \
\mathbb{Z}_{3} \times U_3(5),
 \ \mathbb{Z}_{3}\cdot U_3(5) \}$\\[0.1cm]
$\mathcal{C}_{12}$ & $\!\!\!\!=$  & $ \!\!\!\! \{L:\mathbb{S}_3, \
\mathbb{Z}_{2}\cdot(L:3), \ \mathbb{Z}_{3}\times (L:2), \
\mathbb{Z}_{3}\cdot(L:2), \ (\mathbb{Z}_{2}\times L)\cdot
\mathbb{Z}_{2}$,
\\[0.1cm]
& & \ $(\mathbb{Z}_{3}\cdot L)\cdot \mathbb{Z}_{2}\}$, where
$L=U_3(5)$.\\[0.2cm]
$\mathcal{C}_{13}$ & $\!\!\!\!=$  & $ \!\!\!\! \{{\rm
Aut}(O^+_{10}(2), {\rm Aut}(O^-_{10}(2)\}$,
\\[0.1cm]
\end{tabular}
\newpage
\begin{center}
{\large \bf Acknowledgement}
\end{center}
The first author would like to thank the Research Institute for
Fundamental Sciences (RIFS) for the financial support.

\end{document}